\documentclass{amsart}
\usepackage{amssymb}
\usepackage{amscd}
\usepackage{amsthm}

\numberwithin{equation}{section}
\newtheorem{proposition}{Proposition}[section]
\newtheorem{theorem}[proposition]{Theorem}
\newtheorem{lemma}[proposition]{Lemma}
\newtheorem{corollary}[proposition]{Corollary}
\newtheorem{problem}[proposition]{Problem}
\theoremstyle{definition}
\newtheorem{example}[proposition]{Example}

\newcommand{\lm}{\operatorname{lim}}

\begin{document}
\title[Function spaces and hyperspaces]{A unified theory of function spaces
and hyperspaces: local properties}
\author{Szymon Dolecki}
\address{Mathematical Institute of Burgundy\\
Burgundy University, B.P. 47 870, 21078 Dijon, France}
\author{Fr\'{e}d\'{e}ric Mynard}
\address{Department of Mathematical Sciences, Georgia Southern University,
PB 8093, Statesboro GA 30460, U.S.A.}
\email{dolecki@u-bourgogne.fr}
\email{fmynard@georgiasouthern.edu}
\date{%
\today
}
\maketitle

\begin{abstract}
Every convergence (in particular, every topology) $\tau $ on the hyperspace $%
C\left( X,\$\right) $ \emph{preimage-wise} determines a convergence $\tau
^{\Uparrow }$ on $C\left( X,Z\right) $, where $X,Z$ are topological spaces
and $\$$ is the \emph{Sierpi\'{n}ski topology}, so that $f\in
\lm_{\tau ^{\Uparrow }}\mathcal{F}$ if and only if $f^{-1}(U)\in
\lm_{\tau }\mathcal{F}^{-1}(U)$ for every open subset $U$ of $Z$.
Classical instances are the \emph{pointwise, compact-open and Isbell}
topologies, which are preimage-wise with respect to the topologies, whose
open sets are the collections of, respectively, all (openly isotone) \emph{%
finitely generated}, \emph{compactly generated} and \emph{compact }families
of open subsets of $X$ (compact families are precisely the open sets of the 
\emph{Scott topology}); the \emph{natural }(that is, \emph{continuous})
convergence is preimage-wise with respect to the \emph{natural} \emph{%
hyperspace} convergence.

It is shown that several fundamental local properties hold for a hyperspace
convergence $\tau $ (at the whole space) if and only if they hold for $\tau
^{\Uparrow }$ on $C\left( X,\mathbb{R}\right) $ at the origin, provided that
the underlying topology of $X$ have some $\mathbb{R}$-separation properties.
This concerns character, tightness, fan tightness, strong fan tightness, and
various Fr\'{e}chet properties (from the simple through the strong to that
for finite sets) and corresponds to various covering properties (like Lindel%
\"{o}f, Rothberger, Hurewicz) of the underlying space $X$.

This way, many classical results are unified, extended and improved. Among
new surprising results: the tightness and the character of the natural
convergence coincide and are equal to the Lindel\"{o}f number of the
underlying space; The Fr\'{e}chet property coincides with the Fr\'{e}chet
property for finite sets for the hyperspace topologies generated by compact
networks.
\end{abstract}

\section{Introduction}

The study of the interplay between properties of a topological space $X$ and
those of the associated space $C(X,Z)$ of continuous functions from $X$ to
another topological space $Z$, endowed with convergence structures, is one
of the central themes of topology, and an active area interfacing \emph{%
topology} and \emph{functional analysis}. Most prominent instances for the
space $Z$ is the real line $\mathbb{R}$ (with the usual topology) and a
two-point set $\left\{ 0,1\right\} $ with the \emph{Sierpi\'{n}ski topology} 
$\$:=\left\{ \varnothing ,\left\{ 1\right\} ,\left\{ 0,1\right\} \right\} $ (%
\footnote{%
The present paper focuses as well on these cases.}). By the usual
identifications (\footnote{%
of subsets and maps valued in $\left\{ 0,1\right\} $. More precisely, the
set $C(X,\$)$ is identified with the set of open subsets of $X$, because the
characteristic function of a subset $A$ of $X$, defined by $\chi _{A}(x)=1$
if and only if $x\in A$, is continuous from $X$ to $\$$ if and only if $A$
is open. Of course, $C(X,\$)$ could also be identified with the set of
closed subsets of $X$.}), $C(X,\$)$ becomes the \emph{hyperspace}, either of
open or of closed subsets of $X$.

The \emph{topology of pointwise convergence} (\emph{pointwise topology}, 
\emph{finite-open topology}) is the structure of choice for a sizable share
of such investigations, in part because of its obvious relevance to
functional analysis (weak topologies). The book \cite{Arh.function} of
Arhangel'skii gives a thorough account of this still very active area.

The \emph{compact-open topology} is another most frequently studied
structure; in functional analysis, compatible locally convex topologies are
characterized (via the Arens-Mackey theorem) as those of uniform convergence
on some families of compact sets. The book \cite{McCoy} of McCoy and Ntantu
treats \emph{pointwise convergence} and \emph{compact-open topology}
simultaneously by considering topologies on $C(X,Z)$ with a subbase given by
sets of the form%
\begin{equation}
\lbrack D,U]:=\{f\in C(X,Z):f(D)\subseteq U\},  \label{eq:bracketnotation}
\end{equation}%
where $U$ ranges over open subsets of $Z$ and $D$ ranges over a network $%
\mathcal{D}$ of compact subsets of $X$.

The pointwise topology is the coarsest structure on $C(X,Z)$, for which the
natural coupling%
\begin{equation}
\left\langle \cdot ,\cdot \right\rangle :X\times C(X,Z)\rightarrow Z
\label{natcoupl}
\end{equation}%
is pointwise continuous for each $x\in X$. The \emph{continuous convergence} 
$[X,Z]$ is the coarsest structure on $C(X,Z)$, for which (\ref{natcoupl}) is
jointly continuous. Therefore it satisfies the exponential law (\footnote{%
that is, $[Y,[X,Z]]$ is homeomorphic to $[X\times Y,Z]$ (under $%
^{t}f(y)(x)=f(x,y)$).}) and, as such, has been called \emph{natural
convergence} (e.g., \cite{EscardoLawson}), the terminology that we adopt
here. The exceptional role of the natural convergence among all function
space structures on $C(X,Z)$ was recognized as early as \cite{arensdug} by
Arens and Dugundji, and a compelling case for its systematic use in
functional analysis was made by Binz in \cite{Binz} and more recently and
thoroughly by Beattie and Butzmann in \cite{BB.book}.

As a consequence, even though this paper is mostly focused on\textit{\
completely regular topological spaces }$X$\textit{, }no a priori assumption
is made on the function space convergence structures on $C(X,Z)$. We refer
to \cite{dolecki.BT} for basic terminology and notations on convergence
spaces.

The \emph{Isbell topology} \cite{Isbell75function} was conceived by Isbell
in a hope to provide the topological modification of the natural convergence
on $C(X,Z)$. This is actually the case, when $Z$ is the \emph{Sierpi\'{n}ski
topology} (then the Isbell topology becomes the \emph{Scott topology} on the
lattice of open sets). It is why the Isbell topology plays a central role
when investigating topological spaces from a lattice-theoretic viewpoint 
\cite{compedium}, \cite{contlattices}.

If $\theta $ is a convergence structure (in particular a topology) on $%
C(X,Z),$ we denote by $C_{\theta }(X,Z)$ the corresponding convergence space.

In \cite{D.pannonica} and \cite{groupIsbell}, we studied topologies $\alpha
\left( X,Z\right) $ on $C(X,Z)$ generated by collections $\alpha $ of
families of subsets of $X$ (the $Z$\emph{-dual topology }of $\alpha $), for
which a subbase of open sets consists of%
\begin{equation}
\lbrack \mathcal{A},U]:=\left\{ f\in C\left( X,Z\right) :f^{-}\left(
U\right) \in \mathcal{A}\right\} ,  \label{eq:topbase}
\end{equation}%
where $\mathcal{A}\in \alpha $ and $U$ ranges over the open subsets of $Z$,
and%
\begin{equation*}
f^{-}\left( U\right) :=\left\{ x\in X:f\left( x\right) \in U\right\} 
\end{equation*}%
is our usual shorthand for $f^{-1}\left( U\right) $. If $\alpha $ consists
of (openly isotone) \emph{compact families} (of open subsets of $X$), then $%
\alpha \left( X,Z\right) $ is coarser than the natural convergence $\left[
X,Z\right] $ (that is, is \emph{splitting} according to a widespread
terminology). Pointwise topology, compact-open topology and the Isbell
topology are particular cases of a general scheme (\footnote{%
(\ref{eq:bracketnotation}) is a special case, in which $\mathcal{A}=\mathcal{%
A}_{D}$ is the family of all the open subsets of $X$ that include $D$. Then $%
\alpha =\left\{ \mathcal{A}_{D}:D\in \mathcal{D}\right\} $.}).

The relationship between convergences and topologies on \emph{functional
spaces} $C(X,\mathbb{R})$ and the corresponding convergences and topologies
on \emph{hyperspaces} $C(X,\$)$ (of closed sets or of open sets) is a
principal theme of this paper. Functional spaces and hyperspaces are
intimately related, but also differ considerably for certain aspects. For
instance, the $\mathbb{R}$-dual topologies of collections $\alpha $ of
compact families on completely regular spaces are completely regular, while $%
\$$-dual topologies for the same collections are $T_{0}$ but never $T_{1}$.

Topologies $\alpha \left( X,\$\right) $ and the convergence $\left[ X,\$%
\right] $ have a simpler structure than their counterparts $\alpha \left( X,%
\mathbb{R}\right) $ and $\left[ X,\mathbb{R}\right] $. Actually the
collection $\alpha $ is (itself) a subbabse of open sets of $\alpha \left(
X,\$\right) $. Local properties of $\alpha \left( X,\$\right) $ are
equivalent to some global (covering) properties of $X$ and this equivalence
is usually easily decoded. Therefore in the study of the interdependence
between $X$ and $C(X,\mathbb{R})$, it is essential to comprehend the
relationship between $C(X,\mathbb{R})$ and $C(X,\$)$.

A crucial observation made in \cite{D.pannonica} was that all the mentioned
topologies and convergences on $C(X,\mathbb{R})$ can be characterized
preimage-wise with the aid of the corresponding topologies and convergences
on $C(X,\$).$

In the present paper we unify the investigations of local properties of $C(X,%
\mathbb{R})$ and of $C(X,\$)$ by revealing an abstract connection between
them that embraces all the discussed cases.

As a convention $C_{\tau }(X,\$)$ will always denote the hyperspace of open
subsets of $X$ (endowed with $\tau $) and $cC_{\tau }(X,\$)$ will denote the
homeomorphic image of $C_{\tau }(X,\$)$ under complementation, which is the
corresponding hyperspace of closed subsets of $X$. If $F$ is a set of maps $%
f:X\rightarrow Z$, then $F^{-}(B):=\left\{ f^{-}(B):f\in F\right\} $ and,
for a family $\mathcal{F}$ of sets of such maps, $\mathcal{F}%
^{-}(B):=\left\{ F^{-}(B):F\in \mathcal{F}\right\} .$ We say that $\theta $
is \emph{preimage-wise with respect to} $\tau $ if%
\begin{equation}
f\in \lm_{\theta }\mathcal{F}\Longleftrightarrow \forall _{U\in
C(Z,\$)}\;f^{-}(U)\in \lm_{\tau }\mathcal{F}^{-}(U).
\label{eq:preimage}
\end{equation}

As we shall see, all the topologies $\alpha \left( X,Z\right) $, defined via
(\ref{eq:topbase}), in particular the pointwise, compact-open and Isbell
topologies, as well as the natural convergence $\left[ X,Z\right] $ are
preimage-wise with respect to their hyperspace cases: $\alpha \left(
X,\$\right) $ and $\left[ X,\$\right] $.

This is a special case of the following scheme. Each $h\in C\left(
Z,W\right) $ defines the \emph{lower conjugate} map $h_{\ast }:C\left(
X,Z\right) \rightarrow C\left( X,W\right) $ given by $h_{\ast }\left(
f\right) :=h\circ f$. Each convergence $\tau $ on $C\left( X,W\right) $
determines on $C\left( X,Z\right) $ the coarsest convergence for which $%
h_{\ast }$ is continuous for every $h\in C\left( Z,W\right) $. In the
particular case when $W$ is the Sierpi\'{n}ski topology $\$$, then for each $%
U\in C\left( Z,\$\right) $, the image $U_{\ast }\left( f\right) \in C\left(
X,\$\right) $ and%
\begin{equation}
U_{\ast }\left( f\right) =f^{-}\left( U\right) .  \label{eq:adjoint}
\end{equation}%
In other words, if an element $U$ of $C\left( Z,\$\right) $ is identified
with an open set (via the characteristic function), then in the same way $%
U_{\ast }\left( f\right) $ is identified with the preimage of $U$ by $f$.
Therefore $\theta $ is \emph{preimage-wise with respect to} $\tau $, if the
source%
\begin{equation}
\left( U_{\ast }:C_{\theta }(X,Z)\rightarrow C_{\tau }(X,\$)\right) _{U\in
C(Z,\$)}  \label{eq:generalview}
\end{equation}%
is initial, that is, if $\theta $ is the coarsest convergence on $C(X,Z)$
making each map $U_{\ast }:C(X,Z)\rightarrow C_{\tau }(X,\$)$ continuous (%
\footnote{%
As the category of topological spaces and continuous maps is reflective in
that of convergence spaces (and continuous maps), the coarsest convergence
making the maps $U_{\ast }$ continuous is also the coarsest topology with
this property, whenever $\tau $ is topological.}).

Preimage-wise approach has been implemented in various branches of
mathematics (\footnote{%
Lebesgue says that the idea of preimage-wise study of functions was pivotal
for the theory of his integral \cite{lebesgue}. Greco characterized minmax
properties of real functions in terms of their preimages introducing a
counterpart of measurable sets in analysis \cite{greco.mesurabilite}, \cite%
{greco.decomp}, \cite{greco.minimax}.}). In the study of function spaces,
Georgiou, Iliadis and Papadopoulos in \cite{GIP.dual} considered $Z$-dual
topologies of the type (\ref{eq:topbase}) as well as the topologies on the
set $\left\{ f^{-}\left( U\right) :f\in C\left( X,Z\right) ,U\in C\left(
Z,\$\right) \right\} $ of the form $\left\{ \mathcal{H}^{-}\left( U\right) :%
\mathcal{H}\in \theta ,U\in C\left( Z,\$\right) \right\} $, where $\theta $
is an arbitrary topology on $C\left( X,Z\right) $ (\footnote{%
In general, a topology obtained this way is finer than the restriction of a
topology $\tau $, for which $\theta $ is preimage-wise.}).

The so-called $\gamma $\emph{-connection} of Gruenhage, e.g. \cite%
{Gruenhage.prodfrechet}, is a very particular instance of our preimage-wise
approach (it describes the neighborhood filter of the whole space $X$ for
the pointwise topology on the hyperspace $C\left( X,\$\right) $ of open
sets). Jordan exploited the $\gamma $-connection in \cite{francis.Cp}
establishing a relation between the neighborhood filter of the zero function
in $C_{p}(X,\mathbb{R})$ and the neighborhood filter of the whole space $X$
in $C_{p}(X,\$)$ with the aid of \emph{composable} and \emph{steady}
relations, which enables a transfer of many local properties, like tightness
or character, preserved by such relations.

Jordan's paper is a prefiguration of our theory. Actually the first author
realized that Jordan's approach can be easily extended to general topologies 
$\alpha \left( X,Z\right) $, encompassing, among others, the topology of
pointwise convergence, the compact-open topology and the Isbell topology 
\cite{D.pannonica}. On the other hand, the fact that the natural (or
continuous) convergence fits (\ref{eq:generalview}), that is, that $[X,Z]$
is pre-imagewise with respect to $[X,\$]$, was observed before, e.g. \cite%
{schwarz.powers}.

Yet, even though the relationship between hyperspace structures and function
space structures has been identified on a case by case basis, and even as an
abstract scheme in \cite{GIP.dual} for topologies, it seems that no
systematic use of this situation is to be found in the literature before 
\cite{D.pannonica}. In the present paper, we extend the results of \cite%
{D.pannonica} to general convergences, simplify some of the arguments and
clarify the role of topologicity, and obtain as by-products a wealth of
classical results for function space topologies, as well as new results for
the natural convergence. In particular, we obtain the surprising result that
the character and tightness of the natural convergence on real valued
continuous functions coincide, and are equal to the Lindel\"{o}f degree of
the underlying space.

\section{Preimagewise convergences}

Let $Z$ be a topological space. If $\tau $ is a convergence on $C(X,\$)$,
then $\tau ^{\Uparrow }$ is the convergence on $C(X,Z)$ defined by 
\begin{equation}
f\in \lm_{\tau ^{\Uparrow }}\mathcal{F}\Longleftrightarrow \forall
_{U\in C(Z,\$)}\;f^{-}(U)\in \lm_{\tau }\mathcal{F}^{-}(U).
\label{general-def}
\end{equation}%
In view of (\ref{eq:preimage}), $\tau ^{\Uparrow }$ is preimage-wise with
respect to $\tau $. If for a convergence $\theta $ on $C\left( X,Z\right) $
there exists a convergence on $C\left( X,\$\right) $ with respect to which $%
\theta $ is preimage-wise, then there is a finest convergence $\theta
^{\Downarrow }$ on $C\left( X,\$\right) $ among those $\tau $ for which $%
\theta =\tau ^{\Uparrow }$. Hence, $\tau ^{\Uparrow \Downarrow \Uparrow
}=\tau ^{\Uparrow }$ for each $\tau $.

As we shall show Proposition \ref{pro:idealbasis} below, it is not necessary
to test that $f^{-}(U)\in \lm_{\tau }\mathcal{F}^{-}(U)$ for every
open subset $U$ of $Z$ in (\ref{general-def}), but only for the elements of
an \emph{ideal }(\footnote{%
that is, closed under finite unions}) \emph{basis }of the topology on $Z$.
In terms of closed sets, (\ref{general-def}) becomes%
\begin{equation*}
f\in \lm_{\tau ^{\Uparrow }}\mathcal{F}\Longleftrightarrow \forall
_{C\in cC(Z,\$)}\;f^{-}(C)\in \lm_{c\tau }\mathcal{F}^{-}(C).
\end{equation*}%
In the formula above, analogously to (\ref{general-def}), it is enough test $%
f^{-}(C)\in \lm_{c\tau }\mathcal{F}^{-}(C)$ for the elements of a 
\emph{filtered }(\footnote{%
that is, closed under finite intersections}) \emph{basis (of closed sets)}.

The following is an immediate consequence of the definition.

\begin{proposition}
If $\mathbf{J}$ is a concretely reflective category of convergences, $%
C_{\tau }(X,\$)$ is an object of $\mathbf{J}$, then $C_{\tau ^{\Uparrow
}}(X,Z)$ is also an object of $\mathbf{J}$.
\end{proposition}

In particular, if $\tau $ is a topology, a pretopology or a pseudotopology,
so is $\tau ^{\Uparrow }$.

In the particular important case where $Z=\mathbb{R}$ (\footnote{%
More generally, if $Z$ is perfectly normal.}), the preimage of a closed set
by a continuous function is a \emph{zero set}, because all closed subsets of 
$\mathbb{R}$ are zero sets. Therefore, a $\tau $-preimage-wise\emph{\ }%
convergence on $C(X,\mathbb{R})$ is determined by the restriction of $\tau $
to the cozero sets of $X$ (or the restriction of $c\tau $ to zero sets).
More generally, we say that an open subset $G$ of $X$ is $Z$\emph{%
-functionally open} if there exist $f\in C\left( X,Z\right) $ and $U\in
C\left( Z,\$\right) $ such that $G=f^{-}\left( U\right) $. Of course, all
the elements of $C\left( X,\$\right) $ that are not $Z$\emph{-functionally
open} are isolated for $\tau ^{\Uparrow \Downarrow }$.

\section{Fundamental examples of preimagewise convergences}

Recall that the topology of pointwise convergence as well as the
compact-open topology on $C(X,Z)$ admit subbases of the form $\{[D,U]:U\in
C(Z,\$),D\in \mathcal{D}\}$ where $\mathcal{D}$ is the collection $%
[X]^{<\infty }$ of finite subsets of $X$ in the former case, and the
collection $\mathcal{K}(X)$ of compact subsets of $X$ in the latter, and $%
[D,U]$ is defined by $\left( \ref{eq:bracketnotation}\right) $. We extend
this notation to families of subsets of $X$ by%
\begin{equation*}
\lbrack \mathcal{A},U]:=\bigcup_{A\in \mathcal{A}}[A,U]=\left\{
f:f^{-}\left( U\right) \in \mathcal{A}\right\} .
\end{equation*}%
\newline
If $A$ is a subset of $X$ then $\mathcal{O}_{X}(A)$ denotes the collection
of open subsets of $X$ that contains $A$, and if $\mathcal{A}$ is a
collection of subsets of $X$ then $\mathcal{O}_{X}(\mathcal{A}%
):=\bigcup_{A\in \mathcal{A}}\mathcal{O}_{X}(A)$. A family $\mathcal{A}%
\subseteq C(X,\$)$ is called \emph{compact }if $\mathcal{A}=\mathcal{O}_{X}(%
\mathcal{A})$ and whenever $\mathcal{B}\subseteq C(X,\$)$ such that $%
\bigcup_{B\in \mathcal{B}}B\in \mathcal{A}$, there exists a finite
subcollection $\mathcal{S}$ of $\mathcal{B}$ such that $\bigcup_{B\in 
\mathcal{S}}B\in \mathcal{A}$. The collection $\kappa (X)$ of all compact
families form a topology on $C(X,\$),$ known as the \emph{Scott topology}
(for the lattice of open subsets of $X$ ordered by inclusion) (\footnote{%
The homeomorphic image of $C_{\kappa }(X,\$)$ is the hyperspace $cC_{\kappa
}(X,\$)$ of closed subsets of $X$ endowed with the \emph{upper Kuratowski
topology.}}). The \emph{Isbell topology} on $C(X,Z)$ has a subbase composed
of the sets of the form $[\mathcal{A},U]$ where $U$ ranges over $C(Z,\$)$
and $\mathcal{A}$ ranges over $\kappa (X)$. With the simple observation that%
\begin{equation}
\lbrack \mathcal{O}_{X}(D),U]=[D,U],  \label{eq:Oregular}
\end{equation}%
whenever $U$ is open, one concludes that the topology of pointwise
convergence, the compact-open topology and the Isbell topology are three
instances of function space topologies determined by some $\alpha \subseteq
C(X,\$)$. Indeed, if $\alpha $ is \emph{non-degenerate}, that is, $\alpha
\setminus \varnothing \neq \varnothing ,$ the family%
\begin{equation}
\{[\mathcal{A},U]:\mathcal{A}\in \alpha ,U\in C(Z,\$)\}
\label{funct:subbase}
\end{equation}%
is a subbase for a topology on $C(X,Z),$ denoted $\alpha (X,Z)$. Such
topologies have been called \emph{family-open }in \cite{GIP.dual}. The
corresponding topological space is denoted $C_{\alpha }(X,Z)$.

In view of (\ref{eq:Oregular}) we can, and we will throughout the paper,
assume that each $\mathcal{A}\in \alpha $ is \emph{openly isotone}, that is, 
$\mathcal{A}=\mathcal{O}_{X}(\mathcal{A})$. The topology of pointwise
convergence is obtained when $\alpha $ is the topology $p(X):=\{\bigcup_{F%
\in \mathcal{F}}\mathcal{O}(F):\mathcal{F}\subseteq \lbrack X]^{<\infty }\}$
on $C(X,\$)$ of \emph{finitely generated families}, while the compact open
topology is obtained when $\alpha $ is the topology $k(X):=\{\bigcup_{K\in 
\mathcal{F}}\mathcal{O}(K):\mathcal{F}\subseteq \mathcal{K}(X)\}$ on $%
C(X,\$) $ (\footnote{%
Here, $\mathcal{K}(X)$ stands for the set of all compact subsets of $X$.})
of \emph{compactly generated families}. Of course, the Isbell topology is
obtained when $\alpha $ is the topology $\kappa (X)$ of compact families.

Even if $\alpha \subseteq C(X,\$)$ is not a basis for a topology, $\alpha
(X,Z)=\alpha ^{\cap }(X,Z)$, where $\alpha ^{\cap }$ is the collection of
finite intersections of elements of $\alpha $, because $\bigcap_{i=1}^{n}[%
\mathcal{A}_{i},U]=[\bigcap_{i=1}^{n}\mathcal{A}_{i},U]$. Therefore, \emph{%
we can assume that} $\alpha $ \emph{is a basis for} $\alpha (X,\$)$.

\begin{proposition}
\cite{D.pannonica} If $\alpha \subseteq C(X,\$)$ is non-degenerate then $%
\alpha (X,Z)=\alpha (X,\$)^{\Uparrow }$.
\end{proposition}

\begin{proof}
If $\mathcal{A}\in \alpha $ and $U\in C(Z,\$)$ then 
\begin{equation*}
U_{\ast }^{-1}\left( \mathcal{A}\right) =\{f\in C(X,Z):f^{-}U\in \mathcal{A}%
\}=[\mathcal{A},U]
\end{equation*}%
because $\mathcal{A}=\mathcal{O}_{X}(\mathcal{A})$. Therefore $\alpha (X,Z)$
is indeed the initial topology for the family of maps $\left( U_{\ast
}:C(X,Z)\rightarrow C_{\alpha }(X,\$)\right) _{U\in C(Z,\$)}$.
\end{proof}

By definition, the \emph{natural convergence} $[X,Z]$ on $C(X,Z)$ (also
called continuous convergence, e.g., \cite{Binz}, \cite{BB.book}) is the
coarsest convergence making the canonical coupling (or evaluation) 
\begin{equation}
\langle \cdot ,\cdot \rangle :X\times C(X,Z)\rightarrow Z
\label{nat-coupling}
\end{equation}%
continuous (\footnote{$\langle x,f\rangle :=f(x)$}). In other words, $f\in
\lm_{\lbrack X,Z]}\mathcal{F}$ if and only if for every $x\in X$, the
filter $\langle \mathcal{N}(x),\mathcal{F}\rangle $ converges to $f(x)$ in $%
Z,$ that is, if $U\in \mathcal{O}_{Z}(f(x))$ there is $V\in \mathcal{O}%
_{X}(x)$ and $F\in \mathcal{F}$ such that $\left\langle V,F\right\rangle
\subseteq U,$ equivalently, $F\subseteq \lbrack V,U]$. Therefore

\begin{proposition}
$f_{0}\in \lm_{\lbrack X,Z]}\mathcal{F}$ if and only if for every
open subset $U$ of $Z$ and $x\in X,$ 
\begin{equation}
f_{0}\in \lbrack x,U]\Longrightarrow \exists _{V\in \mathcal{O}_{X}\left(
x\right) }\;[V,U]\in \mathcal{F},  \label{nat}
\end{equation}%
if and only if for every open subset $U$ of $Z$ and $x\in X,$%
\begin{equation}
x\in f_{0}^{-}(U)\Longrightarrow \exists _{F\in \mathcal{F}%
}\;\bigcap_{f\in F}f^{-}(U)\in \mathcal{O}_{X}\left( x\right) .
\label{eq:natcap}
\end{equation}
\end{proposition}

In the case where $Z=\$$, the only non-trivial open subset of $Z$ is $\{1\}$
and elements of $C(X,\$)$ are of the form $\chi _{Y}$ for $Y$ open in $X$.
Therefore $\left( \ref{eq:natcap}\right) $ translates into: $Y\in
\lm_{\lbrack X,\$]}\gamma $ if and only if 
\begin{equation*}
x\in Y\Longrightarrow \exists _{\mathcal{G}\in \gamma
}\;\bigcap_{G\in \mathcal{G}}G\in \mathcal{O}_{X}\left( x\right) ,
\end{equation*}%
In other words, $Y\in \lm_{\lbrack X,\$]}\gamma $ if and only if%
\begin{equation}
Y\subseteq \bigcup_{\mathcal{G}\in \gamma }\mathrm{int}%
_{X}\bigcap_{G\in \mathcal{G}}G.  \label{crit-inh}
\end{equation}

This convergence is often (e.g., \cite{contlattices}) known as the \emph{%
Scott convergence} (in the lattice of open subsets of $X$ ordered by
inclusion). Its homeomorphic image $c[X,\$]$ on the set of closed subsets of 
$X$ is known as \emph{upper Kuratowski convergence} (\footnote{%
Explicitely, if $C$ is a closed subset of $X$ and $\gamma $ is a filter on $%
cC(X,\$)$ then $C\in \lm_{c[X,\$]}\gamma $ if and only if $\bigcap_{%
\mathcal{G}\in \gamma }\mathrm{cl}_{X}\left( \bigcup_{F\in \mathcal{G}%
}F\right) \subseteq C$, that is, $\mathrm{adh}_{X}|\gamma |\subseteq C$
where $|\gamma |:=\left\{ \bigcup_{F\in \mathcal{G}}F:\mathcal{G}\in \gamma
\right\} .$}).

\begin{proposition}
(e.g., \cite{schwarz.powers}) 
\begin{equation*}
\lbrack X,Z]=[X,\$]^{\Uparrow }.
\end{equation*}
\end{proposition}

\begin{proof}
In view of $\left( \ref{eq:natcap}\right) $, $f_{0}\in
\lm_{\lbrack X,Z]}\mathcal{F}$ if and only if $f_{0}^{-}\left(
U\right) \subseteq \bigcup_{F\in \mathcal{F}}\mathrm{int}%
_{X}\bigcap_{f\in F}f^{-}\left( U\right) $, equivalently,%
\begin{equation*}
U_{\ast }(f_{0})\subseteq \bigcup_{\mathcal{G}\in U_{\ast }(%
\mathcal{F})}\mathrm{int}_{X}\bigcap_{G\in \mathcal{G}}G,
\end{equation*}%
for every open subset $U$ of $Z$. In view of $\left( \ref{crit-inh}\right) $%
, we conclude that $f_{0}\in \lm_{\lbrack X,Z]}\mathcal{F}$ if and
only if $U_{\ast }(f_{0})\in \lm_{\lbrack X,\$]}U_{\ast }(\mathcal{F})$ for
every $U\in C(Z,\$)$, which concludes the proof.
\end{proof}

It follows that if $\tau \leq \lbrack X,\$]$ then $\tau ^{\Uparrow }\leq
\lbrack X,Z]$. In other words, the preimage-wise convergence of a splitting
convergence is splitting.

That the natural convergence is not in general topological is a classical
fact and one of the main motivation to consider convergence spaces. It is
well known (see, e.g., \cite{schwarz.powers}, \cite{DGL.kur}) that the
topological reflection $T\left[ X,\$\right] $ of $\left[ X,\$\right] $ is
equal to the \emph{Scott topology }$\kappa \left( X,\$\right) $ and we have
seen that $\kappa \left( X,\$\right) =\kappa (X)$, the collection of all
compact openly isotone families on $X$.

We do not know if for every $X$ there exists a hyperconvergence $\tau $ on $%
C\left( X,\$\right) $ such that $T\left[ X,\mathbb{R}\right] =\tau
^{\Uparrow }$.

\section{Hyperconvergences}

We focus on convergences $\tau $ on $C(X,\$)$ that share basic properties
with $[X,\$]$ and topologies of the type $\alpha (X,\$)$ (\footnote{%
We do not treat here \emph{hit-and-miss} convergences, like the \emph{%
Vietoris topology} or \emph{Fell topology}.}). In particular, we say that $%
\tau $ is \emph{lower} if 
\begin{equation*}
A\subseteq B\in \lm_{\tau }\gamma \Longrightarrow A\in
\lm_{\tau }\gamma ,
\end{equation*}%
and \emph{upper regular} if 
\begin{equation*}
O\in \lm_{\tau }\gamma \Longrightarrow O\in \lm_{\tau }%
\mathcal{O}_{X}^{\natural }(\gamma ),
\end{equation*}%
where $\mathcal{O}_{X}^{\natural }(\gamma )$ is generated by $\{\mathcal{O}%
_{X}(\mathcal{G)}:\mathcal{G}\in \gamma \}$. Observe that if $O_{0},O_{1}$
are open subsets of $Z$ and $\mathcal{F}$ is a filter on $C(X,Z)$ then $%
\mathcal{O}_{X}^{\natural }\left( \mathcal{F}^{-}(O_{0})\right) \leq 
\mathcal{O}_{X}^{\natural }\left( \mathcal{F}^{-}(O_{1})\right) $ whenever $%
O_{0}\subseteq O_{1}$ (\footnote{%
Indeed, if $O_{0}\subset O_{1}$ then $f^{-}(O_{0})\subset f^{-}(O_{1})$,
hence $\bigcup_{f\in F}\left\{ P\in C(X,\$):P\supset
f^{-}(O_{0})\right\} \supset \bigcup_{f\in F}\left\{ P\in
C(X,\$):P\supset f^{-}(O_{1})\right\} .$}). When considering upper regular
convergences, we will often identify a filter $\gamma $ on $C(X,\$)$ and its
upper regularization $\mathcal{O}_{X}^{\natural }(\gamma )$. With this
convention, the previous observation becomes 
\begin{equation}
O_{0}\subseteq O_{1}\Longrightarrow \mathcal{F}^{-}(O_{0})\leq \mathcal{F}%
^{-}(O_{1}).  \label{eq:monotony2}
\end{equation}

\begin{proposition}
\label{pro:up-regular} Each lower topology on $C(X,\$)$ is upper regular.
\end{proposition}

\begin{proof}
It is enough to show that if $\mathcal{G}\subseteq C(X,\$)$ is open then $%
\mathcal{G}=\mathcal{O}_{X}^{\natural }(\mathcal{G})$. Let $A\supseteq G\in 
\mathcal{G}$. Then the principal ultrafilter $A^{\bullet }$ of $A$ converges
to $A$ and therefore to $G$, because the topology is lower. Because $%
\mathcal{G}$ is open, $\mathcal{G}\in $ $A^{\bullet }$ so that $A\in 
\mathcal{G}$. Hence $\mathcal{G}=\mathcal{O}_{X}^{\natural }(\mathcal{G})$.
\end{proof}

\begin{lemma}
\label{lem:closure} If $X\neq \varnothing $ and $p(X,\$)\leq \tau \leq \left[
X,\$\right] $, then 
\begin{equation*}
\mathrm{cl}_{\tau }\left\{ A\right\} =\{O\in C(X,\$):O\subseteq A\}
\end{equation*}%
for each $A\in C(X,\$)$.
\end{lemma}

\begin{proof}
To see that $\mathrm{cl}_{\tau }\left\{ A\right\} =\{O\in
C(X,\$):O\subseteq A\}$, note first that%
\begin{equation*}
\{O\in C(X,\$):O\subseteq A\}\subseteq \mathrm{cl}_{[X,\$]}\left\{
A\right\} \subseteq \mathrm{cl}_{\tau }\left\{ A\right\} \subseteq 
\mathrm{cl}_{p(X,\$)}\left\{ A\right\} ,
\end{equation*}%
where the first inclusion follows from the fact that $[X,\$]$ is lower, and
the others from the assumption $p(X,\$)\leq \tau \leq \left[ X,\$\right] $.
Moreover, if $O\in \mathrm{cl}_{p(X,\$)}\left\{ A\right\} $ then
every $p(X,\$)$-open neighborhood of $O$ contains $A$. In particular, $A\in 
\mathcal{O}_{X}(x)$ for each $x\in O$, so that $O\subseteq A$.
\end{proof}

\begin{proposition}
If $X\neq \varnothing $ and $p(X,\$)\leq \tau \leq \left[ X,\$\right] $,
then $\tau $ is $T_{0}$ but is not $T_{1}$
\end{proposition}

\begin{proof}
By Lemma \ref{lem:closure}, if $A_{1}\neq A_{0}$, say, there is $x\in
A_{1}\setminus A_{0}$, then $\mathrm{cl}_{\tau }\left\{
A_{0}\right\} =\{O\in C(X,\$):O\subseteq A_{0}\}$ is $\tau $-closed and
contains $A_{0}$ but not $A_{1}$ and the convergence is therefore $T_{0}$.
As $X\neq \varnothing $ and $\varnothing \in \mathrm{cl}_{\tau
}\left\{ X\right\} $, the convergence $\tau $ is not $T_{1}$.
\end{proof}

We say that a convergence $\tau $ on $C(X,\$)$ \emph{respects directed sups }%
if whenever $\{\gamma _{i}:i\in I\}$ and $\{B_{i}:i\in I\}$ are two directed
families of filters on $C(X,\$)$ and elements of $C(X,\$)$ respectively,
such that $B_{i}\in \lm_{\tau }\gamma _{i}$ for each $i$ in $I$, we have
that $\bigcup_{i\in I}B_{i}\in \lm_{\tau }\bigvee_{i\in I}\gamma _{i}$. A
compact, lower, upper regular pseudotopology $\tau $ on $C(X,\$)$ that
respects directed sups is called a \emph{solid} \emph{hyperconvergence} (%
\footnote{%
Notions of \emph{upper convergence}, \emph{lower} \emph{regularity} and
respecting directed sups for a convergence on $cC(X,\$)$ are defined dually,
and a compact lower regular upper pseudotopology on $cC(X,\$)$ that respects
directed sups is also called \emph{solid} \emph{hyperconvergence}.}).

Note that in a solid hyperconvergence, every filter converges. Indeed, every
ultrafilter is convergent by compactness, so that every ultrafilter
converges to $\varnothing $ because the convergence is lower. As the
convergence is pseudotopological, \emph{every filter converges to} $%
\varnothing $ \emph{in a solid hyperconvergence.}

\begin{proposition}
$[X,\$]$ and $\alpha (X,\$)$ are solid hyperconvergences provided that $%
\alpha \subseteq \kappa (X)$.
\end{proposition}

\begin{proof}
$[X,\$]$ is well known to be pseudotopological (e.g., \cite{cho}, \cite%
{DM.uK}). In view of $\left( \ref{crit-inh}\right) $, it is lower, and
compact because every filter converges to $\varnothing $. It is upper
regular by Proposition \ref{prop:base}.

It respects directed sups because if $B_{i}\in \lm_{\lbrack X,\$]}\gamma
_{i}$ for each $i\in I$, where the family $\{\gamma _{i}:i\in I\}$ is
directed, then for each $x\in \bigcup_{i\in I}B_{i}$ there is $i$ such that $%
x\in B_{i}\in \lm_{\lbrack X,\$]}\gamma _{i},$ so that there is $\mathcal{G}%
\in \gamma _{i}$ with $x\in \mathrm{int}\left( \bigcap_{G\in 
\mathcal{G}}G\right) $. As $\mathcal{G}\in \gamma _{i}\leq \bigvee_{i\in
I}\gamma _{i},$ we have $\bigcup_{i\in I}B_{i}\subseteq \bigcup_{\mathcal{G}%
\in \bigvee_{i\in I}\gamma _{i}}\mathrm{int}\left( \bigcap_{G\in \mathcal{G}%
}G\right) $.

We have seen that $\alpha (X,\$)\leq \lbrack X,\$]$ whenever $\alpha
\subseteq \kappa (X)$ because $T[X,\$]=\kappa (X,\$)$, so that $\alpha
(X,\$) $ is compact because $[X,\$]$ is. It is lower (and therefore upper
regular by Proposition \ref{pro:up-regular}) because $\mathcal{A}=\mathcal{O}%
_{X}\mathcal{(A)}$ for each $\mathcal{A}\in \alpha $. To see that it
respects directed sups, assume that $B_{i}\in \lm_{\alpha (X,\$)}\gamma
_{i} $ for each $i\in I$, where the families $\{B_{i}:i\in I\}$ and $%
\{\gamma _{i}:i\in I\}$ are directed, and consider $\mathcal{A}\in \alpha $
containing $\bigcup_{i\in I}B_{i}$. By compactness of $\mathcal{A}$ there is
a finite subset $F$ of $I$ such that $\bigcup_{i\in F}B_{i}\in \mathcal{A}$.
Since $\{B_{i}:i\in I\}$ is directed, there is $i_{F}\in I$ such that $%
\bigcup_{i\in F}B_{i}\subseteq B_{i_{F}}\in \mathcal{A}$. Since $%
B_{i_{F}}\in \lm_{\alpha (X,\$)}\gamma _{i_{F}},$ the open set $\mathcal{A}$
belongs to $\gamma _{i_{F}},$ hence to $\bigvee_{i\in I}\gamma _{i}$.
Therefore $\bigcup_{i\in I}B_{i}\in \lm_{\alpha (X,\$)}\bigvee_{i\in
I}\gamma _{i}$.
\end{proof}

\begin{proposition}
\label{pro:idealbasis} If $\tau $ is a solid hyperconvergence, $\mathcal{B}$
is an ideal basis for the topology of $Z,$ and $\mathcal{C}$ is a filtered
basis of closed sets in $Z$, then $f\in \lm_{\tau ^{\Uparrow }}%
\mathcal{F}$ if and only if 
\begin{equation*}
\forall _{B\in \mathcal{B}}\;f^{-}(B)\in \lm_{\tau }\mathcal{F}%
^{-}(B),
\end{equation*}%
if and only if 
\begin{equation*}
\forall _{C\in \mathcal{C}}\;f^{-}(C)\in \lm_{c\tau }\mathcal{F}%
^{-}(C).
\end{equation*}
\end{proposition}

\begin{proof}
We only need to show the first equivalence. Assume that $\forall _{B\in 
\mathcal{B}}\;f^{-}(B)\in \lm_{\tau }\mathcal{F}^{-}(B)$. In view
of $\left( \ref{general-def}\right) $, it is enough to show that $%
f^{-}(O)\in \lm_{\tau }\mathcal{F}^{-}(O)$ whenever $O\in C(Z,\$)$.
Consider a family $\{B_{i}:i\in I\}\subseteq \mathcal{B}$ such that $%
O=\bigcup_{i\in I}B_{i}$. Because $\mathcal{B}$ is an \emph{ideal} basis for
the topology, we can assume this family to be directed, so that $%
\{f^{-}(B_{i}):i\in I\}$ is as well. Moreover, $f^{-}(B_{i})\in
\lm_{\tau }\mathcal{F}^{-}(B_{i})$ for each $i\in I$ and in view
of (\ref{eq:monotony2}), the family of filters $\{\mathcal{F}%
^{-}(B_{i}):i\in I\}$ is directed. Since $\tau $ respects directed sup, 
\begin{equation*}
f^{-}(O)=\bigcup_{i\in I}f^{-}(B_{i})\in \lm_{\tau
}\bigvee_{i\in I}\mathcal{F}^{-}(B_{i}).
\end{equation*}%
Moreover, $\mathcal{F}^{-}(O)\geq \bigvee_{i\in I}\mathcal{F}^{-}(B_{i})$ by 
$\left( \ref{eq:monotony2}\right) $ so that $f^{-}(O)\in \lm_{\tau }%
\mathcal{F}^{-}(O)$, which concludes the proof.
\end{proof}

\section{Interplay between hyperconvergences and the underlying topologies}

Recall that for a family $\mathcal{P}$ of subsets of $X$, we denote $\left\{ 
\mathcal{O}_{X}(P):P\in \mathcal{P}\right\} $ by $\mathcal{O}_{X}^{\natural
}(\mathcal{P})$. Two families $\mathcal{A}$ and $\mathcal{B}$ of subsets of
the same set $X$ \emph{mesh}, in symbols $\mathcal{A}\#\mathcal{B}$, if $%
A\cap B\neq \varnothing $ whenever $A\in \mathcal{A}$ and $B\in \mathcal{B}$%
. We write $A\#\mathcal{B}$ for $\{A\}\#\mathcal{B}$.

\begin{proposition}
\label{prop:refine} The following are equivalent:

\begin{enumerate}
\item $\mathcal{R}\#\mathcal{O}_{X}^{\natural }(\mathcal{P})$ in $C(X,\$);$

\item $\mathcal{P}$ is a refinement of $\mathcal{R}$;

\item $\mathcal{O}_{X}^{\natural }(\mathcal{P})\leq \mathcal{O}%
_{X}^{\natural }(\mathcal{R})$.
\end{enumerate}
\end{proposition}

\begin{proof}
By definition, $\mathcal{R}\#\mathcal{O}_{X}^{\natural }(\mathcal{P})$ if
and only if for each $P\in \mathcal{P}$ there is $R\in \mathcal{R}$ with $%
P\subseteq R$, which means that $\mathcal{P}$ is a refinement of $\mathcal{R}
$. Equivalently, for each $P\in \mathcal{P}$ there is $R\in \mathcal{R}$
such that $\mathcal{O}_{X}(R)\subseteq \mathcal{O}_{X}(P)$, that is, $%
\mathcal{O}_{X}^{\natural }(\mathcal{P})\leq \mathcal{O}_{X}^{\natural }(%
\mathcal{R})$.
\end{proof}

A family $\mathcal{P}$ is said to be an \emph{ideal subbase} if for each
finite subfamily $\mathcal{P}_{0}$ of $\mathcal{P}$ there is $P\in \mathcal{P%
}$ such that $P\supseteq \bigcup \mathcal{P}_{0}$. Note that $\mathcal{O}%
_{X}^{\natural }(\mathcal{P})$ is a filter base if and only if $\mathcal{P}$
is an ideal subbase (\footnote{%
In fact, if $P_{0}\cup P_{1}\subset P$, then $\mathcal{O}_{X}(P_{0})\cup 
\mathcal{O}_{X}(P_{1})=\mathcal{O}_{X}\left( P_{0}\cup P_{1}\right) \supset 
\mathcal{O}_{X}(P)$.}).

If $\gamma $ is a filter on $C(X,\$)$ then%
\begin{equation}
\gamma ^{\Downarrow }:=\left\{ \bigcap_{G\in \mathcal{G}}G:\mathcal{%
G}\in \gamma \right\}  \label{eq:reducedideal}
\end{equation}%
is an ideal subbase of the \emph{reduced ideal of} $\gamma $.

As usual, we extend, in an obvious way, the convergence of filters to that
of their filter bases (\footnote{%
If $\mathcal{B}$ is a filter base and $\tau $ is a convergence, then $y\in
\lm_{\tau }\mathcal{B}$ if $y\in \lm_{\tau }\mathcal{B}%
^{\uparrow }$, where $\mathcal{B}^{\uparrow }$ is the filter generated by $%
\mathcal{B}$.}). A set of filters $\mathbb{B}$ is a \emph{convergence base}
of a convergence $\tau $ on $Y$ if for every $y\in Y$ and each $\mathcal{F}$
with $y\in \lm_{\tau }\mathcal{F}$ there is $\mathcal{B}\in 
\mathbb{B}$ such that $\mathcal{B}\leq \mathcal{F}$ and with $y\in
\lm_{\tau }\mathcal{B}$.

\begin{proposition}
\label{prop:base} $[X,\$]$ admits a convergence base generated by $\mathcal{O%
}_{X}^{\natural }(\mathcal{P})$, where the families $\mathcal{P}$ are ideal
subbases.
\end{proposition}

\begin{proof}
If $Y\in \lm_{\lbrack X,\$]}\gamma $ then, by (\ref{crit-inh}),
the family%
\begin{equation*}
\mathcal{P}:=\left\{ \mathrm{int}\left( A\right) :A\in \gamma ^{\Downarrow
}\right\}
\end{equation*}%
is an open cover of $Y$. Clearly, $\mathcal{P}$ is an ideal subbase, hence $%
\mathcal{O}_{X}^{\natural }(\mathcal{P})$ is a filter base. As for each $%
P\in \mathcal{P}$ there is $\mathcal{G}\in \gamma $ such that $P=\mathrm{int%
}\left( \bigcap_{G\in \mathcal{G}}G\right) $ we infer that $%
\mathcal{G}\subseteq \mathcal{O}_{X}^{\natural }(P)$, that is, $\mathcal{O}%
_{X}^{\natural }(\mathcal{P})$ is coarser than $\gamma $. Finally $Y\in
\lm_{\left[ X,\$\right] }\mathcal{O}_{X}^{\natural }(\mathcal{P})$%
, because $\bigcap \mathcal{O}_{X}^{\natural }(P)=P$ for each $P\in \mathcal{%
P}$, and thus (\ref{crit-inh}) holds.
\end{proof}

\begin{proposition}
\label{pro:adh=lim} If $\mathcal{P}\subseteq C(X,\$)$ is an ideal subbase
and $\tau $ is an upper regular convergence on $C(X,\$)$ then 
\begin{equation*}
\mathrm{adh}_{\tau }\mathcal{P}=\lm_{\tau }\mathcal{O}%
^{\natural }\left( \mathcal{P}\right) .
\end{equation*}
\end{proposition}

\begin{proof}
As $\mathcal{O}^{\natural }\left( \mathcal{P}\right) \#\mathcal{P}$, it is
clear that $\lm_{\tau }\mathcal{O}^{\natural }\left( \mathcal{P}%
\right) \subseteq \mathrm{adh}_{\tau }\mathcal{P}$. Conversely, if 
$U\in \mathrm{adh}_{\tau }\mathcal{P}$ there is a filter $\eta =%
\mathcal{O}^{\natural }(\eta )$ meshing with $\mathcal{P}$ such that $U\in
\lm_{\tau }\eta .$ In other words, for each $\mathcal{A}=\mathcal{O(A)}\in
\eta $ there is $P\in \mathcal{P}\cap \mathcal{A}$. Thus $\mathcal{O}%
(P)\subseteq \mathcal{A}$ and $\mathcal{O}^{\natural }\left( \mathcal{P}%
\right) \geq \eta $ so that $U\in \lm_{\tau }\mathcal{O}^{\natural }\left( 
\mathcal{P}\right) .$
\end{proof}

If $\mathcal{P}\subseteq C\left( X,\$\right) ,$ we denote by $\mathcal{P}%
^{\cup }$ the ideal base generated by $\mathcal{P}$.

\begin{proposition}
Let $\tau $ be a solid hyperconvergence such that $p(X,\$)\leq \tau \leq
\lbrack X,\$]$ and let $\mathcal{P}\subseteq C(X,\$).$ Then $\mathcal{P}$ is
a cover of $U$ if and only if $U\in \mathrm{adh}_{\tau }\mathcal{P}^{\cup }$%
.
\end{proposition}

\begin{proof}
If $\mathcal{P}$ is a cover of $U\ $so is the ideal base $\mathcal{P}^{\cup
} $, so that $U\in \lm_{\lbrack X,\$]}\mathcal{O}^{\natural }\left( 
\mathcal{P}^{\cup }\right) $ by Proposition \ref{prop:base}. Moreover, $%
\mathcal{O}^{\natural }\left( \mathcal{P}^{\cup }\right) \#\mathcal{P}^{\cup
}$ so that $U\in \mathrm{adh}_{[X,\$]}\mathcal{P}^{\cup }\subseteq \mathrm{%
adh}_{\tau }\mathcal{P}^{\cup }.$ Conversely, if $U\in \mathrm{adh}_{\tau }%
\mathcal{P}^{\cup }$ then by Proposition \ref{pro:adh=lim}, $U\in \lm_{\tau
}\mathcal{O}^{\natural }\left( \mathcal{P}^{\cup }\right) \subseteq
\lm_{p(X,\$)}\mathcal{O}^{\natural }\left( \mathcal{P}^{\cup }\right) .$
Therefore, by definition of $p(X,\$),$ for each $x\in U$ there is $S\in 
\mathcal{P}^{\cup } $ such that $\mathcal{O}(S)\subseteq \mathcal{O}(x),$
that is, $x\in S.$ Thus there is $P\in \mathcal{P}$ containing $x$ and $%
\mathcal{P}$ is a cover of $U.$
\end{proof}

\begin{corollary}
\label{cor:idealcover} If $\mathcal{P}\subseteq C(X,\$)$ is an ideal base
and $\tau $ is a solid hyperconvergence such that $p(X,\$)\leq \tau \leq
\lbrack X,\$]$ then%
\begin{equation*}
\mathrm{adh}_{\tau }\mathcal{P}=\lm_{\tau }\mathcal{O}%
^{\natural }\left( \mathcal{P}\right) =\lm_{[X,\$]}\mathcal{O}%
^{\natural }\left( \mathcal{P}\right) =\mathrm{adh}_{[X,\$]}%
\mathcal{P}
\end{equation*}%
consists of those $U\in C\left( X,\$\right) $ for which $\mathcal{P}$ is a
cover of $U$.
\end{corollary}

Corollary \ref{cor:idealcover} does not mean that all the pretopological
solid hyperconvergences between $p(X,\$)$ and $[X,\$]$ coincide! But their
adherences of ideal bases are the same.

\begin{example}
Let $X$ be an infinite countable set with the discrete topology. In this
case $p(X,\$)=\left[ X,\$\right] $. The hyperset $\mathcal{P}:=\left\{
\left\{ x\right\} :x\in X\right\} $ is an open cover of $X$. By definition, $%
Y\in \mathrm{adh}_{p(X,\$)}\mathcal{P}$ if for each finite subset $%
F$ of $Y$ there is $A\in \mathcal{O}_{X}\left( F\right) \cap \mathcal{P}$.
Hence there is $x\in X$ such that $F\subset \left\{ x\right\} $, which means
that the only finite subsets of $Y$ are singletons, that is, $Y$ is a
singleton. On the other hand, $X\in \mathrm{adh}_{p(X,\$)}\mathcal{%
P}^{\cup }$, because $F\in \mathcal{O}_{X}\left( F\right) \cap \mathcal{P}%
^{\cup }$ for each finite subset $F$ of $Y$.
\end{example}

If $\alpha $ is a collection of openly isotone families of subsets of $X,$
we call $\mathcal{P}\subseteq C(X,\$)$ an (open) $\alpha $-\emph{cover} if $%
\mathcal{P}\cap \mathcal{A}\neq \varnothing $ for every $\mathcal{A}\in
\alpha $. Of course, if $p(X)\subseteq \alpha $ then every open $\alpha $%
-cover of $X$ is also an open cover of $X$. Note that the notion of $p(X)$%
-cover coincides with the traditional notion $\omega $-cover, and that the
notion of $k(X)$-cover coincides with the traditional notion $k$-cover (see
e.g., \cite{McCoy}). It follows immediately from the definitions that

\begin{proposition}
\label{pro:adhalpha} Let $\mathcal{P}\subseteq C(X,\$)$ and let $\alpha $ be
a topology on $C(X,\$)$. Then $U\in \mathrm{adh}_{\alpha (X,\$)}\mathcal{P}$
if and only if $\mathcal{P}$ is an $\alpha $-cover of $U$.
\end{proposition}

\section{Transfer of filters}

We shall confer particular attention to the convergence of a filter to the 
\emph{zero function} for the convergence $\tau ^{\Uparrow }$ on $C(X,\mathbb{%
R})$ that is preimage-wise with respect to a solid hyperconvergence $\tau $
on $C(X,\$)$. To that effect, consider a decreasing base of bounded open
neighborhoods of $0$ in $\mathbb{R}$:%
\begin{equation}
\left\{ W_{n}:n<\omega \right\} ,  \label{Wnbase}
\end{equation}%
for instance, let us fix $W_{n}:=\left\{ r\in \mathbb{R}:\left\vert
r\right\vert <\tfrac{1}{n}\right\} $.

\begin{lemma}
\label{lem:conv-at-zero} $\bar{0}\in \lm_{\tau ^{\Uparrow }}%
\mathcal{F}$ if and only if $X\in \lm_{\tau }\mathcal{F}%
^{-}(W_{n}) $ for each $n<\omega $.
\end{lemma}

\begin{proof}
As $\bar{0}^{-}(O)$ is equal either to $X$ (when $0\in O$) or to $%
\varnothing $ (when $0\notin O$), it follows from $(\ref{general-def})$ that
the condition is necessary. Conversely, if an open subset $O$ of $\mathbb{R}$
contains $0$, then there is $n<\omega $ such that $W_{n}\subseteq O$, hence $%
X\in \lm_{\tau }\mathcal{F}^{-}(W_{n})$ implies that $X\in
\lm_{\tau }\mathcal{F}^{-}(O)$, because $\mathcal{F}%
^{-}(W_{n})\leq \mathcal{F}^{-}(O)$. If now $0\notin O$ then $\bar{0}%
^{-}(O)=\varnothing \in \lm_{\tau }\mathcal{F}^{-}(O)$, because $%
\tau $ is a hyperconvergence (hence every filter converges to $\varnothing $%
).
\end{proof}

This special case is important, because it is much easier to compare local
properties of $\tau ^{\Uparrow }$ at $\bar{0}$ with local properties of $%
\tau $ at $X$ than to study analogous properties at an arbitrary $f\in C(X,%
\mathbb{R})$. Moreover, often a study of the mentioned special case is
sufficient for the understanding of this local property at each $f\in C(X,%
\mathbb{R})$. This is feasible whenever all the translations are continuous
for $\tau ^{\Uparrow },$ that is, whenever $\tau ^{\Uparrow }$ is
translation-invariant. It is known that the topology of pointwise
convergence, the compact-open topology, the natural convergence and thus the
natural topology are translation-invariant. Translations are not always
continuous for the Isbell topology (see \cite{DM_Isbell}, \cite%
{FJ_coincidence}), but for each topological space $X$, there exists the
finest translation-invariant topology of the form $\alpha (X,\mathbb{R})$
that is coarser than the Isbell topology $\kappa (X,\mathbb{R})$ \cite%
{groupIsbell}.

Lemma \ref{lem:conv-at-zero} suggests that local properties of $\tau
^{\Uparrow }$ at $\bar{0}$ ``correspond" to local properties of $\tau $ at $X$%
. The remainder of the paper is devoted to making this statement clear and
exploring applications.

If $\alpha $ is a filter on $C(X,\$)$ then, for each (open) subset $W$ of $%
\mathbb{R}$,%
\begin{equation}
\lbrack \alpha ,W]:=\left\{ [\mathcal{A},W]:\mathcal{A}\in \alpha \right\}
\label{W-erected}
\end{equation}%
is a filter base on $C(X,\mathbb{R})$, called the $W$-\emph{erected filter}
of $\alpha $. Note that%
\begin{equation}
\alpha \leq \gamma ,W\supseteq V\Longrightarrow \lbrack \alpha ,W]\leq
\lbrack \gamma ,V].  \label{eq:finer}
\end{equation}

The filter on $C(X,\mathbb{R})$ generated by the filter base $\mathcal{V}$
of the neighborhood filter of $0$ (\footnote{%
Indeed, if $B_{0},B_{1}\in \bigcup_{V\in \mathcal{V}(0)}[\alpha ,V]$%
, then there are $V_{0},V_{1}\in \mathcal{V}(0)$ and $\mathcal{A}_{0},%
\mathcal{A}_{1}\in \alpha $ such that $[\mathcal{A}_{0},V_{0}]\subset B_{0}$
and $[\mathcal{A}_{1},V_{1}]\subset B_{1}$, and thus $[\mathcal{A}_{0}\cap 
\mathcal{A}_{1},V_{0}\cap V_{1}]\subset \lbrack \mathcal{A}_{0},V_{0}]\cap
\lbrack \mathcal{A}_{1},V_{1}]\subset B_{0}\cap B_{1}$.})%
\begin{equation*}
\bigcup_{V\in \mathcal{V}}[\alpha ,V]
\end{equation*}%
does not depend on the choice of a particular neighborhood base of $0$ in $%
\mathbb{R}$ (\footnote{%
In fact, if $\mathcal{V},\mathcal{W}$ are open bases (of the neighborhood
filter of $0$) then for each $W\in \mathcal{W}$ there is $V\in \mathcal{V}$
such that $V\subset W$, hence $[\alpha ,W]\leq \lbrack \alpha ,V]$, and
conversely.}). We denote it by $\left[ \alpha ,\mathcal{N}(0)\right] $ and
call it the \emph{erected filter} of $\alpha $. In particular, if a base is
of the form (\ref{Wnbase}), $[\alpha ,W_{n}]\leq \lbrack \alpha ,W_{n+1}]$
and 
\begin{equation}
\left[ \alpha ,\mathcal{N}(0)\right] =\bigvee_{n<\omega }[\alpha
,W_{n}].  \label{Wn-erected}
\end{equation}

We shall see that if $\alpha $ converges to $X$ in $\tau $ then its erected
filter converges to the null function in $\tau ^{\Uparrow }$. We shall in
fact consider a more general case.

\begin{lemma}
\label{lem:liftsequence} If $\left\{ \alpha _{n}:n<\omega \right\} $ is a
sequence of filters on $C(X,\$)$ such that $\alpha _{n}=O_{X}^{\natural
}(\alpha _{n})$, then the sequence of filters $\left( [\alpha
_{n},W_{n}]\right) _{n<\omega }$ admits a supremum.
\end{lemma}

\begin{proof}
If $S_{1},\ldots ,S_{k}\in \bigcup_{n<\omega }[\alpha _{n},W_{n}]$,
then there are $n_{1},\ldots ,n_{k}<\omega $, say, $n_{1}\leq \ldots \leq
n_{k}$ and $\mathcal{A}_{j}\in \alpha _{n_{j}}$ for $1\leq j\leq k$ such
that $[\mathcal{A}_{j},W_{n_{j}}]\subseteq S_{j}$, and thus $[\mathcal{A}%
,W_{n_{k}}]\subseteq \bigcap_{1\leq j\leq k}[\mathcal{A}%
_{j},W_{n_{j}}]\subseteq $ $\bigcap_{1\leq j\leq k}S_{j}$, where $%
\mathcal{A}:=\bigcap_{1\leq j\leq k}\mathcal{A}_{j}$. As each $%
\alpha _{n}$ is based in openly isotone families $\mathcal{A}$, and $[%
\mathcal{A},W]\neq \varnothing $ provided that $W\neq \varnothing $, the
family $\bigcup_{n<\omega }[\alpha _{n},W_{n}]$ is a filter subbase
and generates $\bigvee_{n<\omega }[\alpha _{n},W_{n}]$.
\end{proof}

\begin{theorem}
\label{thm:anWn} If $X\in \lm_{\tau }\alpha _{n}$ for each $%
n<\omega ,$ then $\bar{0}\in \lm_{\tau ^{\Uparrow
}}\bigvee_{n<\omega }[\alpha _{n},W_{n}]$.
\end{theorem}

\begin{proof}
We use Lemma \ref{lem:conv-at-zero} to check that $\bar{0}\in
\lm_{\tau ^{\Uparrow }}\bigvee_{n<\omega }[\alpha
_{n},W_{n}]$. Let $O$ be an open subset of $\mathbb{R}$ with $0$. Then there
is $n<\omega $ such that $O\supseteq W_{k}$ for $k\geq n$. Then $\left[ 
\mathcal{A},W_{n}\right] ^{-}(O)=\left\{ f^{-}(O):f^{-}(W_{n})\in \mathcal{A}%
\right\} \subseteq \mathcal{A}$ for every $\mathcal{A}$. It follows that, $%
\left[ \alpha _{n},W_{n}\right] ^{-}(O)\geq \alpha _{n}$ so that $X\in
\lm_{\tau }\left[ \alpha _{n},W_{n}\right] ^{-}(O)\subseteq
\lm_{\tau }\left( \bigvee_{n<\omega }[\alpha
_{n},W_{n}]\right) ^{-}(O)$.
\end{proof}

\begin{corollary}
\label{cor:Falpha} If $X\in \lm_{\tau }\alpha $ then $\bar{0}\in
\lm_{\tau ^{\Uparrow }}\left[ \alpha ,\mathcal{N}(0)\right] $.
\end{corollary}

In view of (\ref{eq:monotony2}), we have $\bigvee_{n<\omega }[\alpha
,W_{n}]\leq \left[ \alpha ,\left\{ 0\right\} \right] $ (\footnote{%
Here $f^{-}(0)$ is not open, so that we must use the general definition $%
\left[ \mathcal{A},\left\{ 0\right\} \right] :=\left\{ f:\exists _{A\in 
\mathcal{A}}\;A\subset f^{-}(0)\right\} $.}). Thus:

\begin{corollary}
\label{cor:0-polar} If $X\in \lm_{\tau }\alpha $ then $\bar{0}\in
\lm_{\tau ^{\Uparrow }}\left[ \alpha ,\left\{ 0\right\} \right] $.
\end{corollary}

\section{Construction of classes of filters}

Local properties of a topological space depend on properties of its
neighborhood filters. More generally, local properties of a convergence
space depend on properties of its convergent filters. To understand how
local properties of $\tau $ and $\tau ^{\Uparrow }$ relate, we first need to
understand how the properties of the filter $\alpha $ relate to those of the
filter $\left[ \alpha ,\mathcal{N}(0)\right] $ in Corollary \ref{cor:Falpha}%
. We will explore this question in details in Section \ref%
{sec:transferclasses}. In the present section, we introduce the relevant
terminology, as well as examples of local properties to be considered.

Blackboard letters like $\mathbb{D}$ denote classes of filters, and $\mathbb{%
D}(X)$ denote the set of filters on $X$ of the class $\mathbb{D}$. The class
of principal filters is denoted by $\mathbb{F}_{0}$ and the class of
countably based filters is denoted by $\mathbb{F}_{1}$. More generally, $%
\mathbb{F}_{\kappa }$ stands for the class of filters that admit a base of
cardinality less than $\aleph _{\kappa }$.

A convergence (in particular, a topological) space $X$ is called $\mathbb{D}$%
-\emph{based at} $x$ if whenever $x\in \lm \mathcal{F}$ there is $\mathcal{D%
}\in \mathbb{D}(X),$ $\mathcal{D}\leq \mathcal{F}$ with $x\in \lm \mathcal{D%
}$, and $\mathbb{D}$-\emph{based }if it is $\mathbb{D}$-based at each $x\in
X.$ For example, a convergence (topological) space is first-countable if and
only if it is $\mathbb{F}_{1}$-based.

If $\mathbb{D}$ and $\mathbb{J}$ are two classes of filters, we say that $%
\mathbb{D}$ is $\mathbb{J}$-\emph{steady }if 
\begin{equation*}
\mathcal{D}\in \mathbb{D},\mathcal{J}\in \mathbb{J},\mathcal{D}\#\mathcal{J}%
\Longrightarrow \mathcal{D\vee J}\in \mathbb{D}.
\end{equation*}

As usual, if $R\subseteq X\times Y$ and $D\subseteq X$ then $RD:=\left\{
y\in Y:\exists x\in D,:(x,y)\in R\right\} $ and $\mathcal{RD}:=\left\{
RD:R\in \mathcal{R},D\in \mathcal{D}\right\} $.

A class $\mathbb{D}$ is $\mathbb{J}$-\emph{composable\ }if%
\begin{equation*}
\mathcal{D}\in \mathbb{D}(X),\mathcal{R}\in \mathbb{J}\left( X\times
Y\right) \Longrightarrow \mathcal{RD}\in \mathbb{D}(Y).
\end{equation*}
By convention, we consider that each class $\mathbb{D}$ contains every
degenerate filter. In the sequel, classes that are $\mathbb{F}_{0}$%
-composable and\emph{\ }$\mathbb{F}_{1}$-steady will be of particular
interest.

For each set $X$, we consider the following relations $\lozenge _{\kappa },$ 
$\dag $ and $\triangle $ on $\mathbb{F}(X)$: we write $\mathcal{F}\lozenge
_{\kappa }\mathcal{H}$ if%
\begin{equation*}
\mathcal{\hspace{\stretch{5}}F}\#\mathcal{H}\Longrightarrow \exists A\in
\lbrack X]^{\leq \kappa }:A\#\left( \mathcal{F}\vee \mathcal{H}\right) ;
\end{equation*}

we denote by $\mathcal{F}\triangle _{\kappa }\mathcal{H}$ the following
relation%
\begin{equation*}
\mathcal{F}\#\mathcal{H}\Longrightarrow \exists \mathcal{L}\in \mathbb{F}%
_{\kappa }:\mathcal{L}\geq \mathcal{F}\vee \mathcal{H}.
\end{equation*}

Finally, we write $\mathcal{F}\dag _{1}\mathcal{H}$ if $\mathcal{F}\vee 
\mathcal{H}\in \mathbb{T}_{1}$ where $\mathcal{T}\in \mathbb{T}_{1}$ if 
\begin{equation*}
\left( A_{n}\right) _{n<\omega }\#\mathcal{T}\Longrightarrow \exists
B_{n}\in \lbrack A_{n}]^{<\omega }:\left( \bigcup_{n<\omega
}B_{n}\right) \#\mathcal{T},
\end{equation*}%
and $\mathcal{F}\dag _{0}\mathcal{H}$ if $\mathcal{F}\vee \mathcal{H}\in 
\mathbb{T}_{0}$ where $\mathcal{T}\in \mathbb{T}_{0}$ if 
\begin{equation*}
\left( A_{n}\right) _{n<\omega }\#\mathcal{T}\Longrightarrow \exists
a_{n}\in A_{n}:\{a_{n}:n\in \omega \}\#\mathcal{T}.
\end{equation*}

If $\ast $ is a relation on $\mathbb{F}(X)$ and $\mathbb{D\subseteq F}(X)$,
then $\mathbb{D}^{\ast }:=\{\mathcal{F}\in \mathbb{F}(X):\forall \mathcal{D}%
\in \mathbb{D},$ $\mathcal{F}\ast \mathcal{D}\}$. Many local topological
properties of a space $X$ correspond to the fact that $X$ is $\mathbb{D}%
^{\ast }$-based, for $\mathbb{D}=\mathbb{F}_{0}$ or $\mathbb{D}=\mathbb{F}%
_{1}$.

In particular, a topological space (and by extension, a convergence space)
is respectively \emph{Fr\'{e}chet }(\footnote{%
often called Fr\'{e}chet-Urysohn, but we use the shorter term \emph{Fr\'{e}%
chet}.}), \emph{strongly Fr\'{e}chet}, \emph{productively Fr\'{e}chet}, of $%
\kappa $\emph{-tightness}, \emph{countably fan-tight}, \emph{strongly
countably fan-tight}, if it is $\mathbb{F}_{0}^{\triangle }$-based, $\mathbb{%
F}_{1}^{\triangle }$-based, $\mathbb{F}_{1}^{\triangle \triangle }$-based, $%
\mathbb{F}_{1}^{\lozenge _{\kappa }}$-based, $\mathbb{F}_{1}^{\dag _{1}}$%
-based, $\mathbb{F}_{1}^{\dag _{0}}$-based respectively. Here we gather the
just mentioned equivalences:\smallskip 
\begin{equation}
\begin{tabular}{|c|c|}
\hline
class & based \\ \hline\hline
Fr\'{e}chet & $\mathbb{F}_{0}^{\triangle }$-based \\ \hline
strongly Fr\'{e}chet & $\mathbb{F}_{1}^{\triangle }$-based \\ \hline
productively Fr\'{e}chet & $\mathbb{F}_{1}^{\triangle \triangle }$-based \\ 
\hline
$\kappa $-tight & $\mathbb{F}_{1}^{\lozenge _{\kappa }}$-based \\ \hline
countably fan-tight & $\mathbb{F}_{1}^{\dag _{1}}$-based \\ \hline
strongly countably fan-tight & $\mathbb{F}_{1}^{\dag _{0}}$-based \\ \hline
\end{tabular}%
\bigskip   \label{table}
\end{equation}

Examples of $\mathbb{F}_{0}$-composable and\emph{\ }$\mathbb{F}_{1}$-steady
classes include the class $\mathbb{F}_{n}$ of filters with a filter-base of
cardinality less than $\aleph _{n}$ for $n\geq 1,$ as well as $\mathbb{F}%
_{1}^{\triangle },$ $\mathbb{F}_{1}^{\triangle \triangle },$ $\mathbb{F}%
_{1}^{\lozenge _{\kappa }}$, $\mathbb{F}_{1}^{\dag _{1}}$ and $\mathbb{F}%
_{1}^{\dag _{0}}.$ The class $\mathbb{F}_{0}^{\triangle }$ of \emph{Fr\'{e}%
chet filters} is $\mathbb{F}_{0}$-composable but not\emph{\ }$\mathbb{F}_{1}$%
-steady, and the class $\mathbb{F}_{1}^{\lozenge \lozenge }$ of \emph{%
steadily countably tight filters} is $\mathbb{F}_{1}$-steady but not $%
\mathbb{F}_{0}$-composable. See \cite{mynard-jordan2} for a systematic study
of these concepts and applications to product theorems.

\section{Transfer of classes of filters}

\label{sec:transferclasses}

We notice that the erected filter $\left[ \alpha ,\mathcal{N}(0)\right] $ of 
$\alpha $ can be reconstructed from $\alpha $ with the aid of compositions
of relations as follows. Let $\Delta :=\left\{ \left( f,A,k\right)
:A\subseteq f^{-}(W_{k})\right\} $ and let $\Delta _{j}$ be the $j$-th
projection of $\Delta $. Let $\mathcal{N}$ stand for the cofinite filter on $%
\omega $.

\begin{proposition}
\label{prop:constr-rel}%
\begin{equation*}
\left[ \alpha ,\mathcal{N}(0)\right] =\Delta _{1}(\Delta _{2}^{-}\alpha \vee
\Delta _{3}^{-}\mathcal{N}).
\end{equation*}
\end{proposition}

\begin{proof}
If $\mathcal{A}\in \alpha $ then $\Delta _{2}^{-}\mathcal{A}=\left\{ \left(
f,A,k\right) :f^{-}(W_{k})\in \mathcal{A}\right\} $. If $n<\omega $ then $%
\Delta _{3}^{-}\left\{ k:k\geq n\right\} =\left\{ \left( f,A,k\right)
:\exists _{k\geq n}\;A\subseteq f^{-}(W_{k})\right\} $. Hence%
\begin{equation*}
\Delta _{2}^{-}\mathcal{A}\vee \Delta _{3}^{-}\left\{ k:k\geq n\right\}
=\bigcup_{k\geq n}\left\{ \left( f,A,k\right) :f^{-}(W_{k})\in 
\mathcal{A}\right\} ,
\end{equation*}%
and thus $\Delta _{1}\left( \Delta _{2}^{-}\mathcal{A}\vee \Delta
_{3}^{-}\left\{ k:k\geq n\right\} \right) =\bigcup_{k\geq n}\left\{
f:f^{-}(W_{k})\in \mathcal{A}\right\} =\bigcup_{k\geq n}[\mathcal{A}%
,W_{n}]$. Because $W_{k}\subseteq W_{n}$ if $k\geq n$, hence $[\mathcal{A}%
,W_{k}]\subseteq [\mathcal{A},W_{n}]$. Consequently,%
\begin{equation*}
\Delta _{1}(\Delta _{2}^{-}\alpha \vee \Delta _{3}^{-}\mathcal{N})=\left\{ [%
\mathcal{A},W_{n}]:\mathcal{A}\in \alpha ,n<\omega \right\} =[\alpha ,\left(
W_{n}\right) _{n}].
\end{equation*}
\end{proof}

\begin{corollary}
If $\mathbb{B}$ is an $\mathbb{F}_{0}$-composable and $\mathbb{F}_{1}$%
-steady class of filters and $\alpha \in \mathbb{B}$ then $\left[ \alpha ,%
\mathcal{N}(0)\right] \in \mathbb{B}.$
\end{corollary}

Consider for each $n,$ the relation $[\cdot ,W_{n}]:C(X,\$)\rightarrow C(X,%
\mathbb{R})$. Note that the filter $\bigvee_{n<\omega }[\alpha
_{n},W_{n}]$ of Lemma \ref{lem:liftsequence} is the supremum of the images
of the filters $\alpha _{n}$ under this relation. A class $\mathbb{B}$ of
filters is \emph{countably upper closed }if it is closed under countable
suprema of increasing sequences. In particular:

\begin{proposition}
\label{prop:Bbased} If $\mathbb{B}$ is an $\mathbb{F}_{0}$-composable and
countably upper closed class of filters, and if each $\alpha _{n}\in \mathbb{%
B}$, then $\bigvee_{n<\omega }[\alpha _{n},W_{n}]\in \mathbb{B}$.
\end{proposition}

Let $\mathcal{W}:=\left\{ W_{n}\right\} $ be a fixed base of $\mathcal{N}(0)$
in $\mathbb{R}$. Define%
\begin{equation}
\mathcal{F}^{\mathcal{N}(0)}:=\bigvee_{n<\omega }[\mathcal{F}%
^{-}(W_{n}),W_{n}],  \label{eq:FW}
\end{equation}%
on $C(X,\mathbb{R}),$ associated with a filter $\mathcal{F}$ on $C(X,\mathbb{%
R})$. As $F\subseteq \lbrack F^{-}(W),W]$ for every $F\subset C(X,\mathbb{R})
$ and $W\subset \mathbb{R}$,%
\begin{equation}
\mathcal{F}^{\mathcal{N}(0)}\leq \mathcal{F}.  \label{eq:FWbelow}
\end{equation}

\begin{proposition}
For each symmetric open intervals $V,W$ that contain $0$, there is a
strictly increasing linear map $h$ such that%
\begin{equation*}
\lbrack \mathcal{F}^{-}(W),W]=h\left( [\mathcal{F}^{-}(V),V]\right) .
\end{equation*}
\end{proposition}

\begin{proof}
A base of $[\mathcal{F}^{-}(W),W]$ is of the form%
\begin{equation*}
G_{F}(W):=\left\{ g:g^{-}(W)\in \left\{ f^{-}(W):f\in F\right\} \right\}
:F\in \mathcal{F}
\end{equation*}%
is a base of $[\mathcal{F}^{-}(W),W]$. Let $h$ be a strictly increasing
(linear) map such that $h(V)=W$. Then if $\left( h\circ g\right)
^{-}(W)=g^{-}(V)$. Therefore $g\in G_{F}(V)$ if and only if $h\circ g\in
G_{F}(W)$, that is, $G_{F}(W)=h\left( G_{F}(V)\right) .$
\end{proof}

It follows that, if $W_{n}=r_{n}W$, where $W:=\left( -1,1\right) $ and $%
\left\{ r_{n}\right\} _{n}$ is a decreasing sequence tending to $0$, then%
\begin{equation*}
\mathcal{F}^{\mathcal{N}(0)}=\bigvee_{n<\omega }r_{n}\mathcal{H},
\end{equation*}%
where $\mathcal{H}:=[\mathcal{F}^{-}(W),W]$.

\begin{corollary}
\label{cor:downtoup} Let $\mathbb{B}$ be a class of filters.

\begin{enumerate}
\item If $\mathbb{B}$ is $\mathbb{F}_{0}$-composable and countably upper
closed, and $\tau $ is $\mathbb{B}$-based at $X$, then $\tau ^{\Uparrow }$
is $\mathbb{B}$-based at $\overline{0}$.

\item If $\mathbb{B}$ is $\mathbb{F}_{0}$-composable and $\mathbb{F}_{1}$%
-steady, and $\tau $ is pretopology that is $\mathbb{B}$-based at $X$, then $%
\tau ^{\Uparrow }$ is $\mathbb{B}$-based at $\overline{0}$.
\end{enumerate}
\end{corollary}

\begin{proof}
1. If $\overline{0}\in \lm_{\tau ^{\Uparrow }}\mathcal{F}$ then $X\in
\lm_{\tau }\mathcal{F}^{-}(W_{n})$ for each $n.$ Therefore, for each $n,$
there is $\mathcal{B}_{n}\in \mathbb{B}$ with $X\in \lm_{\tau }\mathcal{B}%
_{n}$ and $\mathcal{B}_{n}\leq \mathcal{F}^{-}(W_{n}).$ In view of Theorem %
\ref{thm:anWn}, $\overline{0}\in \lm_{\tau ^{\Uparrow }}\bigvee_{n<\omega }[%
\mathcal{B}_{n},W_{n}]$. By Proposition \ref{prop:Bbased}, $%
\bigvee_{n<\omega }[\mathcal{B}_{n},W_{n}]\in \mathbb{B}.$ Moreover,%
\begin{equation*}
\bigvee_{n<\omega }[\mathcal{B}_{n},W_{n}]\leq \mathcal{F}^{\mathcal{N}%
(0)}\leq \mathcal{F}\text{,}
\end{equation*}%
which concludes the proof.

2. If $\tau $ is pretopological, then in the proof above, for each $n$ we
can take $\mathcal{B}_{n}=\mathcal{V}_{\tau }(X),$ so that $%
\bigvee_{n<\omega }[\mathcal{B}_{n},W_{n}]=\left[ \mathcal{V}_{\tau }(X),%
\mathcal{N}(0)\right] .$ By Proposition \ref{prop:constr-rel}, $\left[ 
\mathcal{V}_{\tau }(X),\mathcal{N}(0)\right] \in \mathbb{B}$.
\end{proof}

A filter $\alpha $ on $C(X,\$)$ valued in openly isotone families, can be
reconstructed from its erected filter $\left[ \alpha ,\mathcal{N}(0)\right] $
with the aid of compositions of relations, provided that a separation
condition by real-valued continuous functions holds. A family $\mathcal{A}=%
\mathcal{O}_{X}\mathcal{(A)}$ is \emph{functionally separated }if for every $%
O\in \mathcal{A}$, there is $A\in \mathcal{A}$ and $h\in C(X,[0,1])$ such
that $h(A)=\{0\}$ and $h(X\setminus O)=\{1\}$. A hyperfilter is called \emph{%
functionally separated }if it admits a base of functionally separated
hypersets. A solid hyperconvergence on $C(X,\$)$ is \emph{functionally
separated }if whenever $Y\in \lm \gamma $ then there exists $\alpha \leq
\gamma $ such that $Y\in \lm \alpha $ and $\alpha $ is functionally
separated.

It follows from \cite[Lemma 2.5]{DM_Isbell} that compact families on a
completely regular space are functionally separated. Therefore if $\alpha
\subseteq \kappa (X)$ then $\alpha (X,\$)$ is functionally separated.

\begin{lemma}
\label{lem:separation} If $X$ is normal, then $[X,\$]$ is functionally
separated. Moreover, for each bounded open neighborhood $W$ of $0$ in $%
\mathbb{R}$, $[X,\$]$ has a base of filters $\alpha $ such that 
\begin{equation*}
\alpha =\mathcal{O}^{\sharp }\left( \alpha ^{\Downarrow }\right) =\mathcal{O}%
^{\sharp }\left( \left( [\alpha ,\mathcal{N}(0)]^{-}(W)\right) ^{\Downarrow
}\right)
\end{equation*}%
where $\alpha ^{\Downarrow }$ is the reduced ideal of $\alpha $ (\ref%
{eq:reducedideal}).
\end{lemma}

\begin{proof}
In view of (\ref{crit-inh}), if $O\in \lm_{\lbrack X,\$]}\alpha $ then for
each $x\in O$ there exists $\mathcal{A}_{x}\in \alpha $ such that $x\in 
\mathrm{\mathrm{int}}_{X}(\bigcap_{U\in \mathcal{A}_{x}}U)$. By regularity,
there is a closed neighborhood $V_{x}$ of $x$ such that $V_{x}\subseteq 
\mathrm{\mathrm{int}}_{X}(\bigcap_{U\in \mathcal{A}_{x}}U).$ As the family $%
\mathcal{P}:=\{\bigcup_{x\in S}V_{x}:S\in \lbrack O]^{<\infty }\}\ $is an
ideal base, $\mathcal{O}_{X}^{\natural }(\mathcal{P})$ is a filter-base on $%
C(X,\$);$ moreover, $\bigcap_{x\in S}\mathcal{A}_{x}\subseteq \mathcal{O}%
_{X}(\bigcup_{x\in S}V_{x})$ for each $S\in \lbrack O]^{<\infty }$.
Therefore $\alpha \geq \mathcal{O}_{X}^{\natural }(\mathcal{P})$ and $O\in
\lm_{\lbrack X,\$]}\mathcal{O}_{X}^{\sharp }(\mathcal{P}).$ Finally, since $%
\mathcal{P}$ consists of closed sets and $X$ is normal then $\mathcal{O}%
_{X}^{\natural }(\mathcal{P})$ is functionally separated, which completes
the proof.

As shown in the first part of the proof, $[X,\$]$ has a base composed of
filters $\alpha =\mathcal{O}_{X}^{\sharp }\left( \mathcal{P}\right) $ where $%
\mathcal{P}$ is an ideal base of closed sets. For each $P\in \mathcal{P}$,
each $n\in \mathbb{N}$ consider the corresponding element 
\begin{equation*}
R:=\bigcap_{f\in \left[ P,W_{n}\right] }f^{-}\left( W\right)
\end{equation*}%
of $\left( [\alpha ,\mathcal{N}(0)]^{-}(W)\right) ^{\Downarrow }$. Then $%
\mathcal{O}(P)\subseteq \mathcal{O}(R)$ so that $\alpha \geq \mathcal{O}%
\left( ^{\sharp }\left( [\alpha ,\mathcal{N}(0)]^{-}(W)\right) ^{\Downarrow
}\right) $. Indeed, if $\mathcal{O}(P)\nsubseteq \mathcal{O}(R)$ then $%
R\nsubseteq P$ and there is $x\in R\setminus P$. By complete regularity,
there is a continuous map $h\in C(X,\mathbb{R})$ such that $h(x)=1+\sup W$
and $h(P)=\{0\}$. Then $h\in \left[ P,W_{n}\right] $ but $h(R)\nsubseteq W$;
a contradiction.
\end{proof}

Let us call $\$$-\emph{compatible }a class $\mathbb{B}$ of filters
satisfying 
\begin{equation*}
\beta \in \mathbb{B}(C\left( X,\$\right) )\Longrightarrow \mathcal{O}%
^{\sharp }(\beta ^{\Downarrow })\in \mathbb{B}\left( C\left( X,\$\right)
\right) .
\end{equation*}

\begin{theorem}
\label{thm:transfer-compact} If $\alpha $ is a filter on $C(X,\$)$ and $W$
is an open bounded neighborhood of $0$, then $\alpha \leq \left[ \alpha ,%
\mathcal{N}(0)\right] ^{-}\left( W\right) $. If moreover $\alpha $ is
functionally separated, then 
\begin{equation*}
\alpha =\left[ \alpha ,\mathcal{N}(0)\right] ^{-}\left( W\right) .
\end{equation*}
\end{theorem}

\begin{proof}
1. If $n$ is such that $W_{n}\subseteq W$, then $[\mathcal{A}%
,W_{n}]^{-}\left( W\right) \subseteq \mathcal{A}$ for each $\mathcal{A}\in
\alpha $. Indeed, if $G\in \lbrack \mathcal{A},W_{n}]^{-}\left( W\right) $
then there is $A\in \mathcal{A}$ and $f\in C(X,\mathbb{R})$ such that $%
G=f^{-}\left( W\right) $ and $f(A)\subseteq W_{n}$. As $W_{n}\subseteq W$,
we infer that $A\subseteq G,\ $so that $G\in \mathcal{A}$. Consequently $%
\alpha \leq \left[ \alpha ,\mathcal{N}(0)\right] ^{-}\left( W\right)
=\bigvee_{n<\omega }[\mathcal{A},W_{n}]^{-}\left( W\right) $.

2. If $G\in \mathcal{A}$ then, by the functional separation of $\mathcal{A}$%
, there is $A\in \mathcal{A}$ and $h\in C(X,\mathbb{R)}$ such that $%
h(A)=\left\{ 0\right\} $ and $h(X\setminus G)=\left\{ \sup W\right\} $.
Therefore, $h\in \lbrack \mathcal{A},W_{n}]$ for each $n<\omega $, and $%
h^{-}(W)\subseteq G$ so that $G\in \mathcal{O}^{\sharp }\left( [\mathcal{A}%
,W_{n}]^{-}\left( W\right) \right) ,$ hence $\left[ \alpha ,W_{n}\right]
^{-}\left( W\right) \leq \alpha $ for each $n<\omega $.
\end{proof}

In particular, $\alpha \leq \bigwedge_{n<\omega }\left[ \alpha ,%
\mathcal{N}(0)\right] ^{-}\left( W_{n}\right) $ and if if $\alpha $ is
functionally separated, then the equality holds. On the other hand, if $%
\alpha $ is an ultrafilter then $\alpha =\left[ \alpha ,\mathcal{N}(0)\right]
^{-}\left( W\right) $ for any open bounded neighborhood of $0$.

Consider the function $W_{\ast }:C(X,\mathbb{R})\rightarrow C(X,\$)$
(defined by (\ref{eq:adjoint})). It follows from Theorem \ref%
{thm:transfer-compact} that if $\alpha $ is functionally separated, then $%
\alpha =W_{\ast }\left[ \alpha ,\mathcal{N}(0)\right] $, that is, $\alpha $
is the image of $\left[ \alpha ,\mathcal{N}(0)\right] $ by a relation. This
observation constitutes a considerable simplification of a construction
proposed in \cite{francis.Cp} for the finite-open topologies and extended to 
$\alpha $-topologies (\footnote{%
where $\alpha $ is a collection of compact families including all the
finitely generated ones}) in \cite{D.pannonica}.

If $\mathbb{B}$ is a class of filters, let $\mathbb{B}^{^{\bigwedge }}$
denote the class of filters than can be represented as an infimum of filters
of the class $\mathbb{B}$.

\begin{corollary}
\label{cor:uptodown} Let $\mathbb{B}$ be an $\mathbb{F}_{0}$-composable
class of filters.

\begin{enumerate}
\item Let $\tau $ be a functionally separated solid hyperconvergence. If $%
\tau ^{\Uparrow }$ is $\mathbb{B}$-based at $\overline{0}$ then $\tau $ is $%
\mathbb{B}$-based at $X$.

\item If $\tau ^{\Uparrow }$ is $\mathbb{B}$-based at $\overline{0}$ then $%
P\tau $ is $\mathbb{B}^{^{\bigwedge }}$-based at $X.$

\item If $\mathbb{B}$ is $\$$-compatible and $[X,\mathbb{R}]$ is $\mathbb{B}$%
-based at $\overline{0}$, then $[X,\$]$ is $\mathbb{B}$-based at $X$.
\end{enumerate}
\end{corollary}

\begin{proof}
(1). Let $\alpha $ be a functionally separated filter on $C(X,\$)$ such that 
$X\in \lm_{\tau }\alpha .$ By Corollary \ref{cor:Falpha}, $\overline{0}\in
\lm_{\tau ^{\Uparrow }}\left[ \alpha ,\mathcal{N}(0)\right] .$ Therefore,
there is $\mathcal{G}\in \mathbb{B}$ such that $\overline{0}\in \lm_{\tau
^{\Uparrow }}\mathcal{G}$ and $\mathcal{G}\leq \left[ \alpha ,\mathcal{N}(0)%
\right] ,$ hence $X\in \lm_{\tau }\mathcal{G}^{-}(W)$. In view of Theorem %
\ref{thm:transfer-compact},%
\begin{equation*}
\alpha =\left[ \alpha ,\mathcal{N}(0)\right] ^{-}(W)\geq \mathcal{G}^{-}(W),
\end{equation*}%
and $\mathcal{G}^{-}(W)\in \mathbb{B}$ by $\mathbb{F}_{0}$-composability.

(2). If, in the proof above, $\alpha $ is an ultrafilter, then the
assumption of functional separation is not needed. Now the vicinity filter
of $X$ for $P\tau $ is 
\begin{eqnarray*}
\mathcal{V}_{\tau }(X) &=&\bigwedge \left\{ \alpha :\alpha \in \mathbb{U}%
(C\left( X,\$\right) ),X\in \lm_{\tau }\alpha \right\} \\
&=&\bigwedge \left\{ \mathcal{G}^{-}(W_{1}):\alpha \in \mathbb{U}(C\left(
X,\$\right) ),X\in \lm_{\tau }\alpha \right\} .
\end{eqnarray*}

Therefore $\mathcal{V}_{\tau }(X)\in \mathbb{B}^{^{\bigwedge }}$.

(3). In the proof of (1) above, if $\tau =[X,\$]$ then by Lemma \ref%
{lem:separation}, we can assume $\alpha =\mathcal{O}^{\sharp }\left( [\alpha
,\mathcal{N}(0)]^{-}(W)\right) ^{\Downarrow }\geq \mathcal{O}^{\sharp }(%
\mathcal{G}^{-}(W))^{\Downarrow }.$ By $\$$-compatibility and $\mathbb{F}%
_{0} $-composability, $\mathcal{O}^{\sharp }\left( \mathcal{G}^{-}(W)\right)
^{\Downarrow }\in \mathbb{B}$ and $X\in \lm_{\lbrack X,\$]}\mathcal{O}%
^{\sharp }\left( \mathcal{G}^{-}(W)\right) ^{\Downarrow }$ by Proposition %
\ref{prop:base}.
\end{proof}

Combining Corollaries \ref{cor:downtoup} and \ref{cor:uptodown}, we obtain:

\begin{corollary}
Let $\mathbb{B}$ be an $\mathbb{F}_{0}$-composable class of filters.

\begin{enumerate}
\item Let $\tau $ be a functionally separated solid hyperconvergence.

\begin{enumerate}
\item If $\mathbb{B}$ is countably upper closed then $\tau ^{\Uparrow }$ is $%
\mathbb{B}$-based at $\overline{0}$ if and only if $\tau $ is $\mathbb{B}$%
-based at $X$.

\item If $\mathbb{B}$ is $\mathbb{F}_{1}$-steady and if $\tau $ is
pretopological, then $\tau ^{\Uparrow }$ is $\mathbb{B}$-based at $\overline{%
0}$ if and only if $\tau $ is $\mathbb{B}$-based at $X$.
\end{enumerate}

\item If $\mathbb{B}$ is countably upper closed and $\$$-compatible, then $%
[X,\mathbb{R}]$ is $\mathbb{B}$-based if and only if $[X,\$]$ is $\mathbb{B}$%
-based at $X$.
\end{enumerate}
\end{corollary}

In view of Lemma \ref{lem:separation}, we have in particular:

\begin{corollary}
\label{cor:iffinstances} Let $\mathbb{B}$ be an $\mathbb{F}_{0}$-composable
class of filters.

\begin{enumerate}
\item If $\mathbb{B}$ is $\mathbb{F}_{1}$-steady and $\alpha \subseteq
\kappa (X)$, then $C_{\alpha }(X,\mathbb{R})$ is $\mathbb{B}$-based at $%
\overline{0}$ if and only if $C_{\alpha }(X,\$)$ is $\mathbb{B}$-based at $%
X. $

\item If $\mathbb{B}$ is countably upper closed and either $X$ is normal or $%
\mathbb{B}$ is $\$$-compatible, then $[X,\mathbb{R}]$ is $\mathbb{B}$-based
if and only if $[X,\$]$ is $\mathbb{B}$-based at $X.$
\end{enumerate}
\end{corollary}

In particular, if $\mathcal{D}$ is a compact network on a completely regular
space $X$, we consider $\alpha _{\mathcal{D}}:=\mathcal{O}_{X}^{\natural }(%
\mathcal{D)}.$ Then $C_{\alpha _{\mathcal{D}}}(X,\mathbb{R})$ is a
topological group and if $\gamma $ is a cardinal function corresponding to a 
$\mathbb{F}_{1}$-steady and $\mathbb{F}_{0}$-composable class of filters,
like \emph{character }$\chi $, \emph{tightness }$t$, \emph{fan-tightness }$%
\mathrm{vet}$, and \emph{strong fan-tightness} $\mathrm{vet}^{\ast }$,
then 
\begin{equation}
\gamma (C_{\alpha _{\mathcal{D}}}(X,\mathbb{R}))=\gamma (C_{\alpha _{%
\mathcal{D}}}(X,\$),X).  \label{eq:networkcase2}
\end{equation}

As mentioned before, translations need not be continuous for the Isbell
topology on $C(X,\mathbb{R)}$. However, the fine Isbell topology $\overline{%
\kappa }(X,\mathbb{R)}$ is always translation-invariant and the neighborhood
filter of $f$ for the fine Isbell topology is $f+\mathcal{N}_{\kappa }(%
\overline{0})$ \cite[Theorem 4.1]{DM_Isbell}, which implies that the
translations are continuous for the Isbell topology if and only if this
Isbell topology coincides with the fine Isbell topology.

On the other hand, for every $X$ there exists the finest
translation-invariant topology $\Sigma (X,\mathbb{R})$ that is an $\mathbb{R}
$-dual topology of $\Sigma (X)\subseteq \kappa (X)$, hence coarser than the
Isbell topology $\kappa (X,\mathbb{R})$ \cite{groupIsbell}. Therefore%
\begin{equation*}
\gamma (C_{\overline{\kappa }}(X,\mathbb{R}))=\gamma (C_{\kappa }(X,\$),X)%
\text{ and }\gamma (C_{\Sigma }(X,\mathbb{R}))=\gamma (C_{\Sigma }(X,\$),X).
\end{equation*}

We will see in the next section that calculating invariants for $C_{\alpha _{%
\mathcal{D}}}(X,\$),C_{\kappa }(X,\$)$ and $[X,\$]$ in terms of $X$ is often
easy. This way, we will recover a large number of known results, as well as
obtain new ones.

\section{Character and tightness}

\begin{theorem}
\label{th:tightcharacter} (e.g., \cite{mynard.uKc}) The tightness and the
character of $\left[ X,\$\right] $ coincide.
\end{theorem}

\begin{proof}
As the tightness is not greater than character, we need only prove that $%
\chi (\left[ X,\$\right] ,Y)\leq t(\left[ X,\$\right] ,Y)$. Assume that $t(%
\left[ X,\$\right] ,Y)=\lambda $ and let $Y\in \lm_{\left[ X,\$%
\right] }\gamma $. By Proposition \ref{prop:base}, there exists an ideal
subbase $\mathcal{P}$ that is an open cover of $Y$ such that $Y\in
\lm_{\left[ X,\$\right] }\mathcal{O}_{X}^{\natural }(\mathcal{P})$
and $\mathcal{O}_{X}^{\natural }(\mathcal{P})\leq \gamma $. It is clear that 
$\mathcal{P}\#\mathcal{O}_{X}^{\natural }(\mathcal{P})$, hence there is a
family $\mathcal{S}_{0}\subseteq \mathcal{P}$ such that $\mathrm{card}%
\mathcal{S}_{0}\leq \lambda $ and $\mathcal{S}_{0}\#\mathcal{O}%
_{X}^{\natural }(\mathcal{P})$. The family $\mathcal{S}:=\mathcal{S}%
_{0}^{\cup }$ is a subfamily of $\mathcal{P}$, because $\mathcal{P}$ is an
ideal, $\mathrm{card}\mathcal{S}\leq \lambda $ and, a fortiori $\mathcal{S}%
\#\mathcal{O}_{X}^{\natural }(\mathcal{P})$. In view of Proposition \ref%
{prop:refine}, $\mathcal{O}_{X}^{\natural }(\mathcal{P})\leq \mathcal{O}%
_{X}^{\natural }(\mathcal{S})$ and $\mathcal{O}_{X}^{\natural }(\mathcal{S})$
is a filter-base, so that $Y\in \lm_{\left[ X,\$\right] }\mathcal{O%
}_{X}^{\natural }(\mathcal{S})$. Moreover $\mathcal{O}_{X}^{\natural }(%
\mathcal{P})\geq \mathcal{O}_{X}^{\natural }(\mathcal{S})$ because $\mathcal{%
S}\subseteq \mathcal{P}$, so that $\mathcal{O}_{X}^{\natural }(\mathcal{P})=%
\mathcal{O}_{X}^{\natural }(\mathcal{S})$ has a filter base of cardinality
not greater than $\lambda .$
\end{proof}

An immediate consequence of Corollary \ref{cor:idealcover} and Theorem \ref%
{th:tightcharacter} is (the known fact \cite{mynard.uKc}) that 
\begin{equation}
t([X,\$],U)=\chi ([X,\$],U)=L(U)  \label{eq:equalforconthyper}
\end{equation}%
at each $U\in C(X,\$)$, where $L(U)$ is the Lindel\"{o}f degree of $U$.

The $\alpha $-\emph{Lindel\"{o}f number }$\alpha L(U)$ a subset $U$ of $X$
is the smallest cardinal $\lambda $ such that every open $\alpha $-cover of $%
U$ has an $\alpha $-subcover of $U$ of cardinality not greater than $\lambda 
$. In view of Corollary \ref{cor:idealcover}, we have if $p(X)\subseteq
\alpha \subseteq \kappa (X),$ then an ideal base $\mathcal{P}\subseteq
C(X,\$)$ is an open cover of $U\in C(X,\$)$ if and only if it is an $\alpha $%
-cover of $U$. Therefore 
\begin{equation}
L(U)\leq \alpha L(U)  \label{eq:LlessthanalhpaL}
\end{equation}%
for each open subset $U$ of $X$.

It follows immediately from Proposition \ref{pro:adhalpha} that%
\begin{equation}
\alpha L(U)=t(\alpha (X,\$),U).  \label{eq:alphaLind}
\end{equation}

In view of Corollary \ref{cor:iffinstances} (1) and of the fact that the
class $\mathbb{F}_{1}^{\lozenge }$ is $\mathbb{F}_{1}$-steady and $\mathbb{F}%
_{0}$-composable, we obtain:

\begin{theorem}
\label{th:tighntessCalpha} Let $\alpha $ be a topology on $C(X,\$)$ such
that $p(X)\subseteq \alpha \subseteq \kappa (X)$. Then%
\begin{equation*}
\alpha L(X)=t(\alpha (X,\$),X)=t(C_{\alpha }(X,\mathbb{R}),\overline{0}).
\end{equation*}
\end{theorem}

A similar result was announced in \cite[Corollary 3.3]{D.pannonica}, but the
provided proof was not correct. In particular, if $\alpha =\alpha _{\mathcal{%
D}}$ where $\mathcal{D}$ is a network of compact subsets of $X,$ then $%
C_{\alpha _{\mathcal{D}}}(X,\mathbb{R})$ is a topological group and 
\begin{equation}
\alpha _{\mathcal{D}}L(X)=t\left( C_{\alpha _{\mathcal{D}}}(X,\mathbb{R}%
)\right) .  \label{eq:tightalphaD}
\end{equation}

This is exactly \cite[Theorem 4.7.1]{McCoy}. Indeed, in \cite{McCoy}, a $%
\mathcal{D}$-cover of $X$, where $\mathcal{D}$ is a network of closed
subsets of $X$, that is, a family of subsets of $X$ such that every member
of $\mathcal{D}$ is contained in some member of this family. McCoy and
Ntantu define the $\mathcal{D}$-Lindel\"{o}f degree of $X$ as the least
cardinality $\lambda $ such that every open $\mathcal{D}$-cover has a $%
\mathcal{D}$-subcover of cardinality not greater than $\lambda ,$ and
establish that $t\left( C_{\alpha _{\mathcal{D}}}(X,\mathbb{R})\right) $ is
equal to the $\mathcal{D}$-Lindel\"{o}f degree of $X$ \cite[Theorem 4.7.1]%
{McCoy}. It is immediate that the $\mathcal{D}$-Lindel\"{o}f degree of $X$
is equal to $\alpha _{\mathcal{D}}L(X)$. Instances include:

\begin{corollary}
(e.g., \cite[Corollary 4.7.2]{McCoy}) $C_{k}(X,\mathbb{R})$ is countably
tight if and only if every open $k$-cover has a countable $k$-subcover.
\end{corollary}

\begin{corollary}
(e.g., \cite{Arh.function}) The following are equivalent:

\begin{enumerate}
\item $C_{p}(X,\mathbb{R})$ is countably tight;

\item every open $\omega $-cover has a countable $\omega $-subcover;

\item $X^{n}$ is Lindel\"{o}f for every $n\in \omega $.
\end{enumerate}
\end{corollary}

Note that $(2)\Longleftrightarrow (3)$ in the corollary above uses the
observation that 
\begin{equation}
pL(X)=\sup \{L(X^{n}):n\in \omega \},  \label{eq:pL}
\end{equation}%
a proof of which can be found for instance in \cite[Corollary 4.7.3.]{McCoy}.

\begin{proposition}
$\kappa L(U)=t(\kappa (X,\$),U)=t([X,\$],U)=L(U).$
\end{proposition}

\begin{proof}
In view of $T[X,\$]=\kappa (X,\$)$ and of (\ref{eq:equalforconthyper}), 
\begin{equation*}
t(\kappa (X,\$),U)=t(T[X,\$],U)\leq t([X,\$],U)=\chi ([X,\$],U)=L(U),
\end{equation*}%
because $t(X)\geq t(PX)\geq t(TX)$ for any convergence space $X$ \cite[%
Proposition 2.1]{mynard.uKc}. In view of Theorem \ref{th:tighntessCalpha}
and (\ref{eq:LlessthanalhpaL}) 
\begin{equation*}
L(U)\leq \kappa L(U)=t(\kappa (X,\$),U).
\end{equation*}
\end{proof}

In particular $L(X)=\kappa L(X)$, hence for the Isbell topology $\kappa (X,%
\mathbb{R})$ and fine Isbell topology $\overline{\kappa }(X,\mathbb{R})$, we
conclude that

\begin{corollary}
\begin{equation*}
L(X)=t(C_{\kappa }(X,\mathbb{R}),\overline{0})=t(C_{\overline{\kappa }}(X,%
\mathbb{R})).
\end{equation*}
\end{corollary}

It was shown in \cite{BallHager.C(X)} that if $X$ is \v{C}ech-complete then $%
t(C_{k}(X,\mathbb{R}))=L(X).$ We can refine this result as follows (%
\footnote{%
Every \v{C}ech-complete space is consonant \cite[Theorem 4.1]{DGL.kur}, but
not conversely.}):

\begin{corollary}
If $X$ is a (completely regular) consonant topological space then%
\begin{equation*}
t(C_{k}(X,\mathbb{R}))=L(X).
\end{equation*}
\end{corollary}

\begin{proof}
$X$ is consonant if and only if $T[X,\$]=C_{k}(X,\$)$. In view of $\left( %
\ref{eq:networkcase2}\right) $, we have $t(C_{k}(X,\mathbb{R}))=t\left(
C_{k}(X,\$),X\right) $. But $t\left( C_{k}(X,\$),X\right) =t\left(
T[X,\$],X\right) =L(X)$, which concludes the proof.
\end{proof}

The natural convergence $[X,\mathbb{R}]$ is a convergence group, in
particular translation-invariant. Therefore, in view of Corollary \ref%
{cor:iffinstances} (2),%
\begin{equation}
\chi ([X,\mathbb{R}])=\chi ([X,\$],X),
\end{equation}%
because the class $\mathbb{F}_{\lambda }$ is $\$$-compatible, $\mathbb{F}%
_{0} $-composable, and countably upper closed for every cardinal $\lambda $.
Although the class of countably tight filters is not countably upper closed,
we are in a position to see that $t\left( [X,\mathbb{R}]\right) =t([X,\$],X)$%
. Indeed, $t([X,\mathbb{R}])\leq \chi ([X,\mathbb{R}])$ and, in view of
Corollary \ref{cor:uptodown} (2), $t(P[X,\$],X)\leq t([X,\mathbb{R}]),$
because $(\mathbb{F}_{1}^{\lozenge })^{\bigwedge }=\mathbb{F}_{1}^{\lozenge
}.$ Therefore%
\begin{equation*}
L(X)=t(T[X,\$],X)\leq t(P[X,\$],X)\leq t([X,\mathbb{R}])\leq \chi ([X,%
\mathbb{R}])=\chi ([X,\$],X)=L(X).
\end{equation*}

\begin{corollary}
\label{cor:continuousconv} 
\begin{eqnarray*}
L(X) &=&\chi ([X,\$],X)=t([X,\$],X)=t(T[X,\$],X) \\
&=&\chi ([X,\mathbb{R}])=t([X,\mathbb{R}]).
\end{eqnarray*}
\end{corollary}

Note that $L(X)=\chi ([X,\mathbb{R}])$ is a corollary of \cite[Theorem 1]%
{feldman} of Feldman. However, the surprising fact that $\chi ([X,\mathbb{R}%
])=t([X,\mathbb{R}])$ seems to be entirely new.

As we have seen, character and tightness coincide for $[X,\$]$ as well as
for $[X,\mathbb{R}],$ but they do not for $\alpha (X,\$)$ (and therefore not
for $\alpha (X,\mathbb{R})$). By definition the character of $C_{\alpha
}(X,\$)$ at $U$ does not exceed $\lambda $ if there is $\left\{ \mathcal{A}%
_{\beta }:\beta \leq \lambda \right\} \subseteq \alpha $ such that $U\in 
\mathcal{A}_{\beta }$ for each $\beta $ and for each $\mathcal{A}\in \alpha $
such that $U\in \mathcal{A},$ there is $\beta \leq \lambda $ such that $%
\mathcal{A}_{\beta }\subseteq \mathcal{A}$. In particular $\chi (C_{\alpha
}(X,\$),X)\leq \lambda $ if there is a subset $\gamma $ of $\alpha $ of
cardinality at most $\lambda $ such that each element of $\alpha $ contains
an element of $\gamma $. In the particular case where $\alpha =\alpha _{%
\mathcal{D}}$ for a network $\mathcal{D}$ of closed subsets of $X,$ the
condition above translates to: $\chi (C_{\alpha _{\mathcal{D}}}(X,\$),X)\leq
\lambda $ if there is $\mathcal{S}\subseteq \mathcal{D}$ with $|\mathcal{S}%
|\leq \lambda $ such that every element of $\mathcal{D}$ is contained in an
element of $\mathcal{S}$, that is, if $\mathcal{D}$ contains a $\mathcal{D}$%
-cover (in the sense of \cite{McCoy}) of cardinality at most $\lambda .$ In
other words, 
\begin{equation*}
\chi (C_{\alpha _{\mathcal{D}}}(X,\$),X)=\mathcal{D}a(X),
\end{equation*}%
where $\mathcal{D}a(X)$ is the $\mathcal{D}$-Arens number of $X$, as defined
in \cite{McCoy}. In view of Corollary \ref{cor:iffinstances} (1), we recover 
\cite[Theorem 4.4.1]{McCoy}:

\begin{corollary}
If $\mathcal{D}$ is a network of compact subsets of $X$ then:%
\begin{equation*}
\chi (C_{\alpha _{\mathcal{D}}}(X,\mathbb{R}))=\chi (C_{\alpha _{\mathcal{D}%
}}(X,\$),X)=\mathcal{D}a(X).
\end{equation*}
\end{corollary}

Since $C_{\alpha _{\mathcal{D}}}(X,\mathbb{R})$ is a topological group it is
metrizable whenever it is first-countable. Therefore, instances of this
result include that $C_{p}(X,\mathbb{R})$ is metrizable if and only if $X$
is countable, and that $C_{k}(X,\mathbb{R})$ is metrizable if and only if $X$
is hemicompact.

We can more generally define, for $\alpha \subseteq \kappa (X),$ the $\alpha 
$-Arens number $\alpha a(X)$ of $X$ as the least cardinal $\lambda $ such
that there is a subset $\gamma $ of $\alpha $ of cardinality at most $%
\lambda $ such that each element of $\alpha $ contains an element of $\gamma 
$, and we have%
\begin{equation*}
\chi (C_{\alpha }(X,\mathbb{R}),\overline{0})=\chi (C_{\alpha
}(X,\$),X)=\alpha a(X).
\end{equation*}%
The $\alpha $-Arens number seems however somewhat intractable unless $\alpha
=\alpha _{\mathcal{D}}$ for a network $\mathcal{D}$ of closed subsets of $X.$

\section{Fan-tightness and strong fan-tightness}

As the classes of countably fan-tight and strongly countable fan-tight
filters (\ref{table}) are $\mathbb{F}_{1}$-steady and $\mathbb{F}_{0}$%
-composable, Corollary \ref{cor:iffinstances} (1) applies to the effect that%
\begin{eqnarray}
\mathrm{vet}(C_{\alpha }(X,\$),X) &=&\mathrm{vet}(C_{\alpha }(X,\mathbb{R}%
),\overline{0});  \label{eq:vetOtoC} \\
\mathrm{vet}^{\ast }(C_{\alpha }(X,\$),X) &=&\mathrm{vet}%
^{\ast }(C_{\alpha }(X,\mathbb{R}),\overline{0}).  \notag
\end{eqnarray}

It is straightforward from the definitions and Proposition \ref{pro:adhalpha}
that $\mathrm{vet}(C_{\alpha }(X,\$),U)$ (resp. $\mathrm{vet}^{\ast
}(C_{\alpha }(X,\$),X)$) is equal to the minimal cardinality $\lambda $ such
that if for each family $\{\mathcal{P}_{\gamma }:\gamma <\lambda \}$ of open 
$\alpha $-covers of $U$ there are subsets $\mathcal{V}_{\gamma }\subseteq 
\mathcal{P}_{\gamma }$ of cardinality less than $\lambda $ (resp. $P_{\gamma
}\in \mathcal{P}_{\gamma })$ for each $\gamma <\lambda ,$ such that $%
\bigcup_{\gamma <\lambda }\mathcal{V}_{\gamma }$ (resp. $\{P_{\gamma
}:\gamma <\lambda \}$) is an $\alpha $-cover of $U$. Let us call the
cardinal numbers defined above the $\alpha $-\emph{Hurewicz} $\alpha H(X)$
and $\alpha $-\emph{Rothberger} $\alpha R(X)$ \emph{numbers of} $X,$
respectively. In this terminology, we have: 
\begin{eqnarray}
\mathrm{vet}(C_{\alpha }(X,\$),U) &=&\alpha H(U),  \label{eq:hyperH} \\
\mathrm{vet}^{\ast }(C_{\alpha }(X,\$),U) &=&\alpha R(U),
\label{eq:hyperR}
\end{eqnarray}%
for each open subset $U$ of $X$. In particular, \cite[Theorem 1]%
{selectionhyper} and \cite[Theorem 2]{selectionhyper} stating that $%
cC_{p}(X,\$)$ and $cC_{k}(X,\$)$ have countable strong fan-tightness if and
only if $pR(U)=\omega $ and $kR(U)=\omega $ for each open subset $U$ of $X,$
respectively, are instances of $\left( \ref{eq:hyperR}\right) $ for $\alpha
=p(X)$ and $\alpha =k(X).$ Similarly, \cite[Theorem 9]{selectionhyper} and 
\cite[Theorem 10]{selectionhyper} characterizing countable fan-tightness of $%
cC_{p}(X,\$)$ and $cC_{k}(X,\$)$ respectively, are instance of $\left( \ref%
{eq:hyperH}\right) $ for $\alpha =p(X)$ and $\alpha =k(X)$ respectively.

Combining (\ref{eq:vetOtoC}) and $\left( \ref{eq:hyperH}\right) ,$ we have: 
\begin{eqnarray}
\mathrm{vet}(C_{\alpha }(X,\$),X) &=&\mathrm{vet}(C_{\alpha }(X,\mathbb{R}%
),\overline{0})=\alpha H(X);  \label{eq:CH} \\
\mathrm{vet}^{\ast }(C_{\alpha }(X,\$),X) &=&\mathrm{vet}%
^{\ast }(C_{\alpha }(X,\mathbb{R}),\overline{0})=\alpha R(X).
\label{eq:CR}
\end{eqnarray}

Let $s=\{\mathcal{O}(x):x\in X\}$. Note that $C_{s}(X,\$)=C_{p}(X,\$)$. An
infinite topological space $X$ has the Hurewicz property \cite{arh.hurewicz}
(also often called Menger Property, e.g. \cite{selectionhyper}) if and only
if $sH(X):=H(X)=\omega $ and $X$ has the Rothberger property (e.g., \cite%
{millerfremlin}, \cite{scheepers.combI}) if and only if $sR(X):=R(X)=\omega $%
. An argument similar to \cite[Corollary 4.7.3.]{McCoy} was used to show $%
\left( \ref{eq:pL}\right) $ and can be adapted to show that%
\begin{eqnarray}
pH(X) &=&\sup \{H(X^{n}):n\in \omega \};  \label{eq:pH} \\
pR(X) &=&\sup \{R(X^{n}):n\in \omega \}.  \notag
\end{eqnarray}

Note that $\left( \ref{eq:CH}\right) $ particularizes to \cite[Theorem 1]%
{lintightness} when $\alpha =\alpha _{\mathcal{D}}$ where $\mathcal{D}$ is a
network of compact subsets of $X$. Combined with $\left( \ref{eq:pH}\right) $%
, we obtain:

\begin{corollary}
\begin{enumerate}
\item \cite{arh.hurewicz},\cite[Theorem 2]{lintightness} 
\begin{equation*}
\mathrm{vet}(C_{p}(X,\mathbb{R}))=\sup \{H(X^{n}):n\in \omega \},
\end{equation*}%
so that $C_{p}(X,\mathbb{R})$ is countably fan-tight if and only if $X^{n}$
has the Hurewicz property for each $n<\omega $.

\item 
\begin{equation*}
\mathrm{vet}^{\ast }(C_{p}(X,\mathbb{R}))=\sup \{R(X^{n}):n\in
\omega \},
\end{equation*}%
so that $C_{p}(X,\mathbb{R})$ is countably strongly fan-tight if and only if 
$X^{n}$ has the Rothberger property for each $n<\omega $.
\end{enumerate}
\end{corollary}

On the other hand, for $\alpha =k(X)$, we obtain in particular:

\begin{corollary}
\cite{kocinac.Ck}

\begin{enumerate}
\item $C_{k}(X,\mathbb{R})$ is countably fan-tight if and only if for every
sequence $(\mathcal{P}_{n})_{n<\omega }$ of $k$-covers, there are finite
subsets $\mathcal{V}_{n}\subseteq \mathcal{P}_{n}$ for each $n$ such that $%
\bigcup_{n<\omega }\mathcal{V}_{n}$ is a $k$-cover.

\item $C_{k}(X,\mathbb{R})$ is countably strongly fan-tight if and only if
for every sequence $(\mathcal{P}_{n})_{n<\omega }$ of $k$-covers, there are $%
P_{n}\in \mathcal{P}_{n}$ for each $n$ such that $\{P_{n}:n<\omega \}$ is a $%
k$-cover.
\end{enumerate}
\end{corollary}

\section{Fr\'{e}chet properties}

An obstacle to applying the results of Section \ref{sec:transferclasses} to
the Fr\'{e}chet property is that the class of Fr\'{e}chet filters, while $%
\mathbb{F}_{0}$-composable, fails to be $\mathbb{F}_{1}$-steady. The results
apply to the strong Fr\'{e}chet property though, whose associated class of
filters is both $\mathbb{F}_{0}$-composable and $\mathbb{F}_{1}$-steady. We
have seen that tightness and character coincide for $[X,\$]$ and $[X,\mathbb{%
R}]$. Therefore these spaces are Fr\'{e}chet if and only if they are
strongly Fr\'{e}chet if and only if they are countably tight if and only if
they are first-countable. On the other hand,

\begin{theorem}
Let $p\left( X\right) \subseteq \alpha \subseteq \kappa (X)$. The following
are equivalent:

\begin{enumerate}
\item $C_{\alpha }(X,\mathbb{R})$ is strongly Fr\'{e}chet at $\overline{0}$;

\item $C_{\alpha }(X,\$)$ is strongly Fr\'{e}chet at $X$;

\item For every decreasing sequence $\left( \mathcal{P}_{n}\right) _{n\in
\omega }$ of open $\alpha $-covers, for each $n<\omega $ there exists $%
P_{n}\in \mathcal{P}_{n}$ so that each $\mathcal{A}\in \alpha $ contains all
but finitely many of the elements of $\left( P_{n}\right) _{n\in \omega }$.
\end{enumerate}
\end{theorem}

\begin{proof}
The equivalence between (1) and (2) follows from Corollary \ref%
{cor:iffinstances} (1), and the equivalence between (2) and (3) follows
immediately from the definition of strongly Fr\'{e}chet and Proposition \ref%
{pro:adhalpha}.
\end{proof}

The Fr\'{e}chet property for function spaces can nevertheless be
characterized with our results in the special case of $\alpha =\alpha _{%
\mathcal{D}}$ for a network $\mathcal{D}$.

Following \cite{Gruenhage05}, we call a topological space $X$ \emph{Fr\'{e}%
chet-Urysohn for finite sets at }$x\in X$, or FU$_{fin}$ at $x$, if for any $%
\mathcal{P}\subseteq \lbrack X]^{<\infty }$ such that each $U\in \mathcal{O}%
_{X}(x)$ contains an element of $\mathcal{P}$, there is a sequence $%
(P_{n})_{n\in \omega }\subseteq \mathcal{P}$ such that each $U\in \mathcal{O}%
_{X}(x)$ contains all but finitely many elements of $(P_{n})_{n\in \omega }$%
. We call a filter $\mathcal{F}$ an \emph{FU}$_{fin}$\emph{-filter }if for
any $\mathcal{P}\subseteq \lbrack X]^{<\infty }$ such that $\mathcal{P}\geq 
\mathcal{F}$, there is a sequence $(P_{n})_{n\in \omega }\subseteq \mathcal{P%
}$ such that $(P_{n})_{n\in \omega }\geq \mathcal{F}$. Let $\mathbb{FU}%
_{fin} $ denote the corresponding class of filters. Clearly, a space is FU$%
_{fin}$ at $x$ if it is $\mathbb{FU}_{fin}$-based at $x$.

\begin{theorem}
\label{thm:F=Ffin} Let $\mathcal{D}$ be a network of compact subsets of $X$
and $Y\in C\left( X,\$\right) $. If $C_{\alpha _{\mathcal{D}}}(X,\$)$ is Fr%
\'{e}chet at $Y$ then $C_{\alpha _{\mathcal{D}}}(X,\$)$ is Fr\'{e}%
chet-Urysohn for finite sets at $Y$.
\end{theorem}

\begin{proof}
Let $\beta $ be a family of finite subsets of $C(X,\$)$ such that for each $%
D\in \mathcal{D}$ containing $Y$, there is $\mathcal{P}\in \beta $ with $%
\mathcal{P}\subseteq \mathcal{O}_{X}(D)$. In other words, $D\subseteq
\bigcap_{P\in \mathcal{P}}P$. Since the intersection is finite, $%
\bigcap_{P\in \mathcal{P}}P\in \mathcal{O}_{X}(D)$. Therefore, $Y\in 
\mathrm{cl}_{\alpha _{\mathcal{D}}}\left\{ \bigcap_{P\in \mathcal{P}}P:%
\mathcal{P}\in \beta \right\} $. As $C_{\alpha _{\mathcal{D}}}(X,\$)$ is Fr%
\'{e}chet at $Y$, there is a sequence $(\mathcal{P}_{n})_{n\in \omega }\ $of
elements of $\beta $ such that $Y\in \lm_{\alpha _{\mathcal{D}}}\left(
\bigcap_{P\in \mathcal{P}_{n}}P\right) _{n\in \omega }$. In other words, for
each $Y\subseteq D\in \mathcal{D}$, there is $n_{D}$ such that $%
\bigcap_{P\in \mathcal{P}_{n}}P\in \mathcal{O}_{X}(D)$ for each $n\geq
n_{D}, $ so that $\mathcal{P}_{n}\subseteq \mathcal{O}_{X}(D)$ for each $%
n\geq n_{D} $, which proves that $C_{\alpha _{\mathcal{D}}}(X,\$)$ is FU$%
_{fin}$ at $Y$.
\end{proof}

The method of the proof does not work for general topologies $\alpha \left(
X,\$\right) $ with $\alpha \subseteq \kappa \left( X\right) $, because
compact families do not need to be filters. In particular, there remains the
following problem (of course, for dissonant $X$):

\begin{problem}
Does the Fr\'{e}chet property and the FU$_{fin}$ property coincide for the
Scott topology $C_{\kappa }(X,\$)$?
\end{problem}

It is known (e.g., \cite{ReczSipa99}) that a FU$_{fin}$ topological space is 
$\alpha _{2}$ (in the sense of \cite{ARA} \footnote{%
A topological space $X$ has property $\alpha _{2}$ (at $x$) if for each
sequence $\left( \sigma _{n}\right) _{n\in \omega }$ of sequences converging
to $x$, there is a sequence $\sigma $ convergent to $x$ such that for each $%
n\in \omega $, the set $\sigma _{n}\cap \sigma $ is infinite.}). Therefore
Theorem \ref{thm:F=Ffin} implies that in $C_{\alpha _{\mathcal{D}}}(X,\$)$
the Fr\'{e}chet property implies $\alpha _{2}$, and \emph{a fortiori} $%
\alpha _{3}$ and $\alpha _{4}$, in particular implies the strong Fr\'{e}chet
property.

\begin{lemma}
The class $\mathbb{FU}_{fin}$ is $\mathbb{F}_{0}$-composable and $\mathbb{F}%
_{1}$-steady.
\end{lemma}

\begin{proof}
Let $\mathcal{F}\in \mathbb{FU}_{fin}(X),$ $A\subseteq X\times Y$ and let $%
\mathcal{P}\subseteq \lbrack Y]^{<\infty }$ such that $\mathcal{P}\geq A%
\mathcal{F}$. In other words, for each $F\in \mathcal{F}$ there is $P_{F}\in 
\mathcal{P}$ such that $P_{F}\subseteq AF$. Hence for each $y\in P_{F}$
there is $x_{y}\in F$ such that $(x_{y},y)\in A$. Let $Q_{F}:=\{x_{y}:y\in
P_{F}\}$ and let $\mathcal{Q}:=\{Q_{F}:F\in \mathcal{F}\}$. Then $\mathcal{Q}%
\subseteq \lbrack X]^{<\infty }$ such that $\mathcal{Q}\geq \mathcal{F}$.
Therefore there is a sequence $(F_{n})_{n\in \omega }\subseteq \mathcal{F}$
such that $(Q_{F_{n}})_{n\in \omega }\geq \mathcal{F}$. It is easy to see
that $(P_{F_{n}})_{n\in \omega }\geq A\mathcal{F}$, which shows that $%
\mathbb{FU}_{fin}$ is $\mathbb{F}_{0}$-composable.

The class $\mathbb{FU}_{fin}$ is $\mathbb{F}_{0}$-steady because if $%
\mathcal{P}\geq A\vee \mathcal{F}$ there is $\mathcal{P}_{0}\subseteq 
\mathcal{P}$ such that $\mathcal{P}_{0}\geq A\vee \mathcal{F}$ and $\mathcal{%
P}_{0}\subseteq \lbrack A]^{<\infty }$. Moreover, by \cite{ReczSipa99} or 
\cite[Theorem 20]{Gruenhage05}, $\mathbb{FU}_{fin}\times \mathbb{F}%
_{1}\subseteq \mathbb{FU}_{fin}$ (in terms of of \cite{mynard-jordan2}),
hence \cite[Theorem 20(1)]{mynard-jordan2}, $\mathbb{FU}_{fin}$ is therefore
also $\mathbb{F}_{1}$-steady.
\end{proof}

\begin{theorem}
\label{th:FrechetCalpha} Let $\mathcal{D}$ be a network of compact subsets
of $X$. The following are equivalent:

\begin{enumerate}
\item $C_{\alpha _{\mathcal{D}}}(X,\$)$ is Fr\'{e}chet at $X$;

\item $C_{\alpha _{\mathcal{D}}}(X,\$)$ is FU$_{fin}$ at $X$;

\item $C_{\alpha _{\mathcal{D}}}(X,\mathbb{R})$ is FU$_{fin}$;

\item $C_{\alpha _{\mathcal{D}}}(X,\mathbb{R})$ is Fr\'{e}chet;

\item For every open $\mathcal{D}$-cover $\mathcal{C}$ of $X$, there exists
a countable subfamily $\mathcal{S}$ of $\mathcal{C}$ such that every $D\in 
\mathcal{D}$ is contained in all but finitely many elements of $\mathcal{S}$.
\end{enumerate}
\end{theorem}

\begin{proof}
$(1)\Longleftrightarrow (2)$ follows from Theorem \ref{thm:F=Ffin}. $%
(1)\Longleftrightarrow (5)$ follows immediately from the definitions. $%
(2)\Longleftrightarrow (3)$ follows from Corollary \ref{cor:iffinstances}
(1), because the class of FU$_{fin}$ filters is $\mathbb{F}_{0}$-composable
and $\mathbb{F}_{1}$-steady and $C_{\alpha _{\mathcal{D}}}(X,\mathbb{R})$ is
a topological group. $(3)\Longrightarrow (4)$ and $(2)\Longrightarrow (1)$
are obvious, and $(4)\Longrightarrow (1)$ follows from Corollary \ref%
{cor:uptodown} (2), because $\mathbb{F}_{0}^{\triangle }=(\mathbb{F}%
_{0}^{\triangle })^{^{\bigwedge }}$.
\end{proof}

Note that the equivalence $(4)\Longleftrightarrow (5)$ is \cite[Theorem 4.7.4%
]{McCoy}. In the case $\alpha _{\mathcal{D}}=p(X),$ the equivalences $%
(3)\Longleftrightarrow (4)\Longleftrightarrow (5)$ are due to \cite%
{gerlits.Cp}.

The case $\alpha _{\mathcal{D}}=p(X)$ generalizes \cite[Proposition 5 (1)]%
{convhyper} stating that $cC_{p}(X,\$)$ is $\alpha _{2}$ whenever it is Fr%
\'{e}chet. On the other hand, when $\alpha _{\mathcal{D}}=k(X)$, \cite[%
Proposition 5 (2)]{convhyper} is generalized in two ways: we only need to
assume that $cC_{k}(X,\$)$ is Fr\'{e}chet (rather than the more stringent
condition of strict Fr\'{e}chetness) and we obtain that $cC_{k}(X,\$)$ is FU$%
_{fin}$ rather than $\alpha _{2}$.

Note however that while the Fr\'{e}chet property is equivalent to
sequentiality and even to being a $k$-space for $C_{p}(X,\mathbb{R})$ and $%
C_{k}(X,\mathbb{R})$ (e.g., \cite{Pytkeev}), these properties are not
equivalent for the corresponding hyperspaces. For instance, an example of a
space $X$ for which $C_{k}(X,\$)$ is sequential but not Fr\'{e}chet is given
in \cite[p. 275]{CHV}. Therefore, the results of Section \ref%
{sec:transferclasses} in general do not apply to sequentiality.

\section{Appendix: dual convergences}

We have seen that each non-degenerate $\alpha \subseteq C(X,\$)$ composed of
openly isotone families defines a $Z$\emph{-dual topology} $\alpha (X,Z)$ on 
$C(X,Z)$ via (\ref{funct:subbase}). Note that $f\in \lm_{\alpha
(X,Z)}\mathcal{F}$ if and only if%
\begin{equation}
\forall _{O\in \mathcal{O}_{Z}}\forall _{\mathcal{A}\in \alpha }\;f\in
\lbrack \mathcal{A},O]\Longrightarrow \lbrack \mathcal{A},O]\in \mathcal{F}.
\label{dual-top}
\end{equation}

In view of the characterization $\left( \ref{nat}\right) $ of the natural
convergence, it is natural to consider for each collection $\alpha $ of
(openly isotone) families on $X$ the $Z$\emph{-dual convergence }$[\alpha
,Z] $ defined by: $f\in \lm_{\lbrack \alpha ,Z]}\mathcal{F}$ if
and only if%
\begin{equation}
\forall _{O\in \mathcal{O}_{Z}}\forall _{\mathcal{A}\in \alpha }\;f\in
\lbrack \mathcal{A},O]\Longrightarrow \exists _{A\in \mathcal{A}}\;[A,O]\in 
\mathcal{F}.  \label{dual-conv}
\end{equation}

Distinct collections $\alpha $ of families of open sets generate distinct
topologies on $C(X,Z)$ provided that the elements of $C(X,Z)$ separate these
families in $X$. Such a separation is assured for example by the $Z$%
-regularity of $X$ and the compactness of the elements of $\alpha $ (see 
\cite[Proposition 2.1]{groupIsbell}). In contrast, all the collections $%
\alpha $ including $p\left( X\right) $ and included in $\kappa \left(
X\right) $ give rise the same convergence, which turns out to be the \emph{%
natural convergence}.

\begin{theorem}
\label{th:equality} The dual convergence $[\alpha ,Z]$ is equal to the
natural convergence $[X,Z]$ for each collection $\alpha $ such that $%
p(X)\leq \alpha \leq \kappa (X)$.
\end{theorem}

\begin{proof}
We first show that $[X,Z]\geq \lbrack \kappa (X),Z].$ To this end, assume
that $f_{0}\in \lm_{\lbrack X,Z]}\mathcal{F}$ and let $f_{0}\in \lbrack 
\mathcal{A},O]$ where $O$ is $Z$-open and $\mathcal{A}\in \kappa (X)$. It
follows that $f_{0}^{-}\left( O\right) \in \mathcal{A}$. If $x\in
f_{0}^{-}\left( O\right) $ then there is $V_{x}\in \mathcal{O}(x)$ such that 
$V_{x}\subseteq f_{0}^{-}\left( O\right) $ and $\left[ V_{x},O\right] \in 
\mathcal{F}$. By the compactness of $\mathcal{A}$, there is a finite subset $%
B$ of $f_{0}^{-}\left( O\right) $ such that $V:=\bigcup_{x\in
B}V_{x}\in \mathcal{A}$. On the other hand, $\left[ V,O\right]
=\bigcap_{x\in B}\left[ V_{x},O\right] \in \mathcal{F}$ showing
that $f_{0}\in \lm_{\lbrack \kappa \left( X\right) ,Z]}\mathcal{F}$.

As $[\kappa (X),Z]\geq \lbrack p(X),Z]$, it is now enough to show that $%
[p(X),Z]\geq \lbrack X,Z]$. Suppose that $f_{0}\in \lm_{\lbrack
p\left( X\right) ,Z]}\mathcal{F}$ and let $x\in X,O\in \mathcal{O}_{Z}$ be
such that $f_{0}\in \lbrack x,O]$, equivalently $f_{0}^{-}\left( O\right)
\in \mathcal{O}_{X}\left( x\right) $, or else, $f_{0}\in \left[ \mathcal{O}%
_{X}\left( x\right) ,O\right] $. By the assumption, there is $V\in \mathcal{O%
}_{X}\left( x\right) $ such that $[V,O]\in \mathcal{F}$, that is, $f_{0}\in
\lm_{\lbrack X,Z]}\mathcal{F}$.
\end{proof}

Note that, since $[A,O]\subseteq \lbrack \mathcal{A},O]$ for each $A\in 
\mathcal{A}$,%
\begin{equation}
\lbrack \alpha ,Z]\geq T[\alpha ,Z]\geq \alpha (X,Z).  \label{eq:3alphas}
\end{equation}

\bibliographystyle{amsplain}

\begin{thebibliography}{10}

\bibitem{arensdug}
R.~Arens and J.~Dugundji, \emph{Topologies for function spaces}, Pacific J.
  Math. \textbf{1} (1951), 5--31.

\bibitem{ARA}
A.~V. Arhangel'skii, \emph{The frequency spectrum of a topological space and
  the classification of spaces}, Math. Dokl. \textbf{{\bf 13}} (1972),
  1185--1189.

\bibitem{Arh.function}
\bysame, \emph{Topological function spaces}, Kluwer Academic, Dordrecht, 1992.

\bibitem{arh.hurewicz}
A.V. Arhangel'skii, \emph{Hurewicz spaces, analytic sets and fan tightness of
  function spaces}, Soviet Math. Dokl. \textbf{33} (1986), 396--399.

\bibitem{BallHager.C(X)}
R.N. Ball and A.W. Hager, \emph{Network character and tightness of the
  compact-open topology}, Comment. Math. Univ. Carolin. \textbf{47} (2006),
  no.~3, 473--482.

\bibitem{BB.book}
R.~Beattie and H.~P. Butzmann, \emph{Convergence {S}tructures and
  {A}pplications to {F}unctional {A}nalysis}, Kluwer Academic, 2002.

\bibitem{Binz}
E.~Binz, \emph{Continuous convergence in {$C(X)$}}, Springer-Verlag, 1975,
  Lect. Notes Math. 469.

\bibitem{CHV}
L.~Hol{\`a} C.~Costantini and P.~Vitolo, \emph{Tightness, character and related
  properties of hyperspace topologies}, Top. Appl. \textbf{142} (2004),
  245--292, to appear.

\bibitem{cho}
G.~Choquet, \emph{Convergences}, Ann. Univ. Grenoble \textbf{{\bf 23}}
  (1947-48), 55--112.

\bibitem{convhyper}
Giuseppe Di~Maio, Lj. D.~R. Ko{\v{c}}inac, and Tsugunori Nogura,
  \emph{Convergence properties of hyperspaces}, J. Korean Math. Soc.
  \textbf{44} (2007), no.~4, 845--854. \MR{MR2334529 (2008k:54008)}

\bibitem{GIP.dual}
Georgiou D.N., Illiadis S.D., and Papadopoulos B.K., \emph{On dual topologies},
  Top. Appl. \textbf{140} (2004), no.~1, 57--68.

\bibitem{D.pannonica}
S.~Dolecki, \emph{Properties transfer between topologies on function spaces,
  hyperspaces and underlying spaces}, Mathematica Pannonica \textbf{19} (2008),
  no.~2, 243--262.

\bibitem{dolecki.BT}
\bysame, \emph{An initiation into convergence theory}, Contemporary Mathematics
  486, vol. Beyond Topology, pp.~115--161, A.M.S., 2009.

\bibitem{DGL.kur}
S.~Dolecki, G.~H. Greco, and A.~Lechicki, \emph{When do the upper {K}uratowski
  topology (homeomorphically, {S}cott topology) and the cocompact topology
  coincide?}, Trans. Amer. Math. Soc. \textbf{{\bf 347}} (1995), 2869--2884.

\bibitem{groupIsbell}
S.~Dolecki, F.~Jordan, and F.~Mynard, \emph{Group topologies coarser than the
  {I}sbell topology}, to appear.

\bibitem{DM_Isbell}
S.~Dolecki and F.~Mynard, \emph{When is the {I}sbell topology a group
  topology?}, to appear in Top. Appl.

\bibitem{DM.uK}
\bysame, \emph{Hyperconvergences.}, Appl. Gen. Top. \textbf{4} (2003), no.~2,
  391--419.

\bibitem{EscardoLawson}
Mart{\'{\i}}n Escard{\'o}, Jimmie Lawson, and Alex Simpson, \emph{Comparing
  {C}artesian closed categories of (core) compactly generated spaces}, Topology
  Appl. \textbf{143} (2004), no.~1-3, 105--145. \MR{MR2080286 (2005f:54047)}

\bibitem{feldman}
W.~A. Feldman, \emph{Axioms of countability and the algebra {$C(X)$}}, Pacific
  J. Math. \textbf{{\bf 47}} (1973), 81--89.

\bibitem{gerlits.Cp}
J.~Gerlits and Z.~Nagy, \emph{Some properties of {$C(X)$},}, Top. Appl.
  \textbf{14} (1982), 151--161.

\bibitem{compedium}
G.~Gierz, K.~H. Hofmann, K.~Keimel, J.~Lawson, M.~Mislove, and D.~Scott,
  \emph{A compedium of continuous lattices}, Springer-Verlag, Berlin, 1980.

\bibitem{contlattices}
G.~Gierz, K.H. Hofmann, K.~Keimel, J.~Lawson, M.~Mislove, and D.~Scott,
  \emph{Continuous lattices and domains}, Encyclopedia of Mathematics, vol.~93,
  Cambridge University Press, 2003.

\bibitem{greco.mesurabilite}
G.~Greco, \emph{Sur la mesurabilit{\'e} d'une fonction num{\'e}rique par
  rapport {\`a} une famille d'ensembles}, Rend. Semin. Mat. Univ. Podova
  \textbf{65} (1981), 163--176.

\bibitem{greco.decomp}
\bysame, \emph{Decomposizioni di semifiltri e $\gamma$-limiti sequenziali in
  reticoli completamente distributivi}, Ann. Mat. pura e appl. \textbf{137}
  (1984), 61--82.

\bibitem{greco.minimax}
\bysame, \emph{Minimax theorems and saddling transformations}, J. Math. An.
  Appl. \textbf{147} (1990), 180--197.

\bibitem{Gruenhage.prodfrechet}
G.~Gruenhage, \emph{Products of {F}r{\'e}chet spaces}, Top. Proc. \textbf{30}
  (2006), no.~2, 475--499.

\bibitem{Gruenhage05}
G.~Gruenhage and P.~Szeptycki, \emph{Fr\'echet-{U}rysohn for finite sets},
  Topology Appl. \textbf{151} (2005), no.~1-3, 238--259. \MR{MR2139755
  (2006i:54004)}

\bibitem{Isbell75function}
J.~Isbell, \emph{Function spaces and adjoints}, Math. Scand. \textbf{36}
  (1975), 317--339.

\bibitem{francis.Cp}
F.~Jordan, \emph{Productive local properties of function spaces}, Top. Appl.
  \textbf{154} (2007), no.~4, 870--883.

\bibitem{FJ_coincidence}
\bysame, \emph{Coincidence of function space topologies}, Top. Appl.
  \textbf{157} (2010), no.~2, 336--351.

\bibitem{mynard-jordan2}
F.~Jordan and F.~Mynard, \emph{Compatible relations of filters and stability of
  local topological properties under supremum and product}, Topology Appl.
  \textbf{{\bf 153}} (2006), 2386--2412.

\bibitem{kocinac.Ck}
Lj.D.R. Ko{\v c}inac, \emph{Closure properties of function spaces}, Appl. Gen.
  Top. \textbf{4} (2003), no.~2, 255--261.

\bibitem{lebesgue}
H.~Lebesgue, \emph{Sur le d{\'e}veloppement de la notion d'int{\'e}grale}, Mat.
  Tidsskrift \textbf{B} (1926), 54--74.

\bibitem{lintightness}
S.~Lin, \emph{Tightness of function spaces}, Appl. Gen. Top. \textbf{7} (2006),
  no.~1, 103--107.

\bibitem{selectionhyper}
G.~Di Maio, Lj.D.R. Ko{\v c}inac, and E.~Meccariello, \emph{Selection
  principles and hyperspace topologies}, Top. Appl. \textbf{153} (2005),
  912--923.

\bibitem{McCoy}
R.~A. McCoy and I.~Ntantu, \emph{Topological properties of spaces of continuous
  functions}, Springer-Verlag, 1988.

\bibitem{millerfremlin}
A.W. Miller and D.H. Fremlin, \emph{On some properties of {H}urewicz, {M}enger
  and {R}othberger}, Fund. Math. \textbf{129} (1988), 17--33.

\bibitem{mynard.uKc}
F.~Mynard, \emph{First-countability, sequentiality and tightness of the upper
  {K}uratowski convergence}, Rocky Mountain J. of Math. \textbf{33} (2003),
  no.~3, 1011--1038.

\bibitem{Pytkeev}
E.~G. Pytkeev, \emph{On the sequentiality of spaces of continuous functions},
  Communications Moscow Math. Soc. (1982), 190--191.

\bibitem{ReczSipa99}
E.~A. Reznichenko and O.~V. Sipacheva, \emph{Properties of {F}r\'echet-{U}ryson
  type in topological spaces, groups and locally convex spaces}, Vestnik
  Moskov. Univ. Ser. I Mat. Mekh. (1999), no.~3, 32--38, 72. \MR{MR1711871
  (2000h:54006)}

\bibitem{scheepers.combI}
M.~Scheepers, \emph{Combinatorics of open covers {I}: {R}amsey {T}heory}, Top.
  Appl. \textbf{69} (1996), 31--62.

\bibitem{schwarz.powers}
F.~Schwarz, \emph{Powers and exponential objects in initially structured
  categories and application to categories of limits spaces}, Quaest. Math.
  \textbf{{\bf 6}} (1983), 227--254.

\end{thebibliography}

\end{document}